\documentclass[11pt]{amsart}
\usepackage[T1]{fontenc}
\usepackage[utf8]{inputenc}
\usepackage[
  backend=biber,
  style=numeric,
  sorting=none,
  giveninits=false,
  doi=true,
  isbn=true,
  url=true,
  eprint=true
]{biblatex}
\usepackage{lmodern}
\usepackage{microtype}
\usepackage{tabularx}
\usepackage{array}
\usepackage{amsmath,amssymb,amsfonts,amsthm}
\usepackage{mathtools}
\usepackage{enumitem}
\usepackage{tikz-cd}
\usepackage{geometry}
\usepackage{float}

\usepackage[colorlinks=true,linkcolor=blue,citecolor=blue,urlcolor=blue]{hyperref}
\usepackage[nameinlink,noabbrev]{cleveref}

\hypersetup{
  pdftitle={From Finite-Node Conifold Geometry to BPS Structures, I: Algebraic State Data},
  pdfauthor={Abdul Rahman}
}
\numberwithin{table}{section}
\DeclareMathOperator{\Sing}{Sing}
\DeclareMathOperator{\Cone}{Cone}

\numberwithin{equation}{section}

\theoremstyle{plain}
\newtheorem{theorem}{Theorem}[section]
\newtheorem{proposition}[theorem]{Proposition}
\newtheorem{lemma}[theorem]{Lemma}
\newtheorem{corollary}[theorem]{Corollary}

\theoremstyle{definition}
\newtheorem{definition}[theorem]{Definition}

\theoremstyle{remark}
\newtheorem{remark}{Remark}[section]

\makeatletter



\makeatother

\newcommand{\Q}{\mathbb{Q}}

\newcommand{\C}{\mathbb{C}}

\newcommand{\Perv}{\mathrm{Perv}}
\newcommand{\Hom}{\mathrm{Hom}}
\newcommand{\Ext}{\mathrm{Ext}}

\newcommand{\var}{\mathrm{var}}
\newcommand{\rat}{\mathrm{rat}}

\title[From Finite-Node Conifold Geometry to BPS Structures I]{From Finite-Node Conifold Geometry to BPS Structures I: Algebraic State Data}
\author{Abdul Rahman}
\thanks{Email: arahman@alum.howard.edu}

\begin{document}

\begin{abstract}
Let $\pi:X\to \Delta$ be a one-parameter degeneration whose central fiber $X_0$ is a complex threefold with finitely many ordinary double points $\Sigma=\{p_1,\dots,p_r\}\subset X_0$. Associated with this degeneration is the corrected finite-node perverse extension, together with its mixed-Hodge-module refinement and a finite-node schober datum whose perverse-sheaf shadow is identified with the corrected perverse sheaf $\mathcal P$. The purpose of the present paper is to extract from these finite-node geometric, extension-theoretic, mixed-Hodge, and categorical inputs the intrinsic algebraic state data carried by the degeneration. More precisely, we isolate the finite localized quotient $Q_\Sigma:=\bigoplus_{k=1}^r i_{k*}\Q_{\{p_k\}}$, the nodewise coupling space $E_\Sigma:=\Ext^1_{\Perv(X_0;\Q)}(Q_\Sigma,IC_{X_0})$, its canonical nodewise decomposition $E_\Sigma\cong\bigoplus_{k=1}^r \Q e_k$, and the coefficient vector $c_\Sigma=(c_1,\dots,c_r)\in\Q^r$ defined by $[\mathcal P]_{\mathrm{perv}}=\sum_{k=1}^r c_k e_k$. We then prove that these state variables are compatible with both the mixed-Hodge-module lift and the schober realization of $\mathcal P$, so that the same finite-node architecture appears simultaneously in perverse, mixed-Hodge, and categorical form. The resulting package $(V_\Sigma,E_\Sigma,c_\Sigma)$ is the intrinsic algebraic state data attached to the finite-node conifold degeneration. It provides the first algebraic layer in the passage from finite-node geometry to later incidence, quiver, stability, BPS-spectral, and wall-crossing structures.
\end{abstract}
\maketitle
\section{Introduction}

Let $\pi:X\to \Delta$ be a one-parameter degeneration whose central fiber \(X_0\) has finitely many ordinary double points $\Sigma=\{p_1,\dots,p_r\}\subset X_0$. The basic corrected object associated with this degeneration is the perverse sheaf
\[
\mathcal P:=\Cone(\var_F)[-1],
\qquad
F:=\Q_X[3],
\]
constructed from the nearby-/vanishing-cycle formalism. In the finite-node case, the corrected object fits into the exact sequence
\begin{equation}\label{eq:intro-corrected-extension}
0\longrightarrow IC_{X_0}\longrightarrow \mathcal P\longrightarrow Q_\Sigma\longrightarrow 0,
\qquad
Q_\Sigma:=\bigoplus_{k=1}^r i_{k*}\Q_{\{p_k\}}.
\end{equation}
The finite singular contribution is therefore a direct sum of point-supported rank-one sectors, one for each node. At the same time, the corrected global extension class is not merely a formal sum of local pieces: it is controlled by a distinguished nodewise extension package
\begin{equation}\label{eq:intro-nodewise-ext}
\Ext^1_{\Perv(X_0;\Q)}(Q_\Sigma,IC_{X_0})
\cong
\bigoplus_{k=1}^r
\Ext^1_{\Perv(X_0;\Q)}\!\bigl(i_{k*}\Q_{\{p_k\}},IC_{X_0}\bigr)
\cong
\bigoplus_{k=1}^r \Q e_k,
\end{equation}
and the global corrected class admits a unique expansion
\begin{equation}\label{eq:intro-coefficient-vector}
[\mathcal P]_{\mathrm{perv}}
=
\sum_{k=1}^r c_k e_k,
\qquad
c_\Sigma:=(c_1,\dots,c_r)\in \Q^r.
\end{equation}
These finite node-indexed data are the intrinsic algebraic state variables extracted from the degeneration. The point of the present paper is to isolate these state variables systematically and prove that they are compatible across the three realizations already established in the earlier papers: \textit{perverse}, \textit{mixed-Hodge}, and \textit{categorical}. On the mixed-Hodge side, the corrected extension lifts to an object
\[
\mathcal P^H\in MHM(X_0)
\]
fitting into the exact sequence
\begin{equation}\label{eq:intro-mhm-extension}
0\longrightarrow IC^H_{X_0}
\longrightarrow \mathcal P^H
\longrightarrow
\bigoplus_{k=1}^r i_{k*}\Q^H_{\{p_k\}}(-1)
\longrightarrow 0,
\qquad
\rat(\mathcal P^H)\cong \mathcal P.
\end{equation}
On the categorical side, the finite-node schober datum
\begin{equation}\label{eq:intro-schober}
S_\Sigma=
\Bigl(
\mathcal C_{\mathrm{bulk}},
\{\mathcal C_{p_k}\}_{k=1}^r,
\{\Phi_k,\Psi_k\}_{k=1}^r,
Sh(S_\Sigma)
\Bigr)
\end{equation}
has one localized sector \(\mathcal C_{p_k}\) for each node \(p_k\in\Sigma\) and shadow
\[
Sh(S_\Sigma)\cong \mathcal P.
\]
Thus the same finite-node architecture appears simultaneously as a corrected perverse extension, a mixed-Hodge-module extension, and a finite-node schober shadow.

\subsection{Why finite-node algebraic state data}

We observe that the finite-node conifold setting already has the formal shape of a finite system of localized sectors coupled to a bulk geometric sector. The singular quotient
\[
Q_\Sigma=\bigoplus_{k=1}^r i_{k*}\Q_{\{p_k\}}
\]
is a finite direct sum of point-supported rank-one terms; the nodewise extension decomposition \eqref{eq:intro-nodewise-ext} isolates a distinguished one-dimensional coupling line at each node; and the corrected global extension class is encoded by the finite coefficient vector \(c_\Sigma\). On the mixed-Hodge side, the same finite localized quotient appears as
\[
\bigoplus_{k=1}^r i_{k*}\Q^H_{\{p_k\}}(-1),
\]
while on the categorical side the same node set indexes the localized schober sectors \(\mathcal C_{p_k}\). The natural problem and role of the present paper is therefore not to invent an algebraic language from outside, but to extract the finite algebraic state variables already implicit in the geometry.  Yet our task here is more basic and more rigid: to prove that a finite-node conifold degeneration canonically determines intrinsic algebraic state data
\[
(V_\Sigma,E_\Sigma,c_\Sigma)
\]
compatible with the corrected perverse extension, its mixed-Hodge lift, and its finite-node schober shadow. Note that a stability formalism, a wall-crossing formalism, or a full BPS-spectrum calculation is not treated here but left for future work.

\subsection{Relation to earlier work}

The starting point is the corrected perverse object constructed in the ordinary double point case. In the single-node setting, that object was defined from nearby and vanishing cycles and shown to be the canonical perverse extension carrying the point-supported rank-one singular contribution. The next step was to place that construction into the nearby-cycle and limiting-mixed-Hodge-theoretic setting, thereby identifying the common origin of the corrected perverse singular term and the vanishing contribution in the limiting mixed Hodge structure. In the finite-node setting, the corrected perverse extension was then lifted to a mixed Hodge module
\[
\mathcal P^H\in MHM(X_0)
\]
with quotient
\[
\bigoplus_{k=1}^r i_{k*}\Q^H_{\{p_k\}}(-1),
\]
and the nodewise organization of the global corrected extension class was made explicit. Finally, the finite-node schober paper constructed a schober datum
\[
S_\Sigma=
\Bigl(
\mathcal C_{\mathrm{bulk}},
\{\mathcal C_{p_k}\}_{k=1}^r,
\{\Phi_k,\Psi_k\}_{k=1}^r,
Sh(S_\Sigma)
\Bigr)
\]
with one localized categorical sector per node and shadow equal to \(\mathcal P\). See \cite{RahmanSchoberPaper,RahmanPerverseNearbyCycles,RahmanMultiNodeSchoberPaper, RahmanMixedHodgeModules}. The present paper extracts from these earlier results the intrinsic finite algebraic state variables attached to the degeneration. In that sense, it is the first of the algebraic papers:
\[
\text{geometry}
\;\longrightarrow\;
\text{state data}
\;\longrightarrow\;
\text{incidence/quiver data}
\;\longrightarrow\;
\text{BPS spectra and wall crossing}.
\]

\subsection{Primitive local anchoring and the role of $e_k$}

The distinguished nodewise generators
\[
e_k \in \Ext^1_{\Perv(X_0;\Q)}\!\bigl(i_{k*}\Q_{\{p_k\}}, IC_{X_0}\bigr)
\]
should not be viewed merely as a convenient basis of the extension space
\[
E_\Sigma \cong \bigoplus_{k=1}^r \Q e_k.
\]
Rather, they should be read as the first algebraic shadows of the nodewise primitive local package already visible in the corrected perverse, mixed-Hodge, and categorical realizations of the finite-node degeneration. On the perverse side, they arise from the decomposition of the corrected finite-node extension space associated with
\[
0 \to IC_{X_0} \to P \to \bigoplus_{k=1}^r i_{k*}\Q_{\{p_k\}} \to 0.
\]
On the mixed-Hodge side, the same nodewise architecture is reflected by the quotient
\[
\bigoplus_{k=1}^r i_{k*}\Q^H_{\{p_k\}}(-1),
\]
while on the categorical side the same finite node set labels the localized schober sectors
\[
\{C_{p_k}\}_{k=1}^r.
\]

Thus the present paper should be read as isolating the first algebraic package attached to a primitive local nodewise structure already visible across the perverse, mixed-Hodge, and categorical realizations of the finite-node degeneration. The later incidence, quiver, stability, BPS, and wall-crossing papers will not introduce unrelated algebraic labels from outside the geometry; rather, they will organize dynamically the same primitive nodewise directions first isolated here. In this sense, the state-data package extracted below is already geometrically anchored: the vectors $e_k$ are the first algebraic carriers of the primitive local sectors attached to the nodes of the degeneration.

\subsection{Role of this work in the larger sequence}

This work is the first algebraic extraction paper in the finite-node program. Its role is to isolate the intrinsic state variables attached to the corrected finite-node package before any interaction, stability, or BPS formalism is introduced. Concretely, the output of the present paper is the package
\[
(V_\Sigma, E_\Sigma, c_\Sigma),
\]
consisting of the finite vertex set, the distinguished nodewise coupling space, and the coefficient vector of the corrected global extension class.

This is the algebraic layer required before one can ask how localized sectors interact, how an incidence or quiver-theoretic package should be assembled, or how later stability and BPS data should be organized. In particular, the present paper does not yet study the functorial coupling pattern carried by the attachment functors of the finite-node schober package, and it does not yet pass to incidence relations, graded interaction laws, stability data, chamber structures, or BPS indices. Those later structures require that the finite-node state variables first be isolated in a theorem-level way.

Accordingly, the present paper should be read as the state-variable paper of the sequence. It identifies the intrinsic finite algebraic variables canonically determined by the degeneration and thereby supplies the minimal algebraic input for the later incidence/quiver, stability, BPS, and wall-crossing papers.

\subsection{Relation to the BPS/quiver literature}

In the BPS-quiver literature, quivers and related algebraic data serve as the finite algebraic interface through which one studies basic sectors, interaction channels, stability, and wall crossing. Geometric input is used to determine BPS quivers or quiver-type data, and the BPS spectrum is then analyzed through the resulting representation-theoretic or quantum-mechanical problem; see, for example, \cite{AlimCecottiCordovaEspahbodiRastogiVafa_BPSQuivers,AlimCecottiCordovaEspahbodiRastogiVafa_N2Quivers,Denef_QuantumQuivers,Cecotti_QuiverBPS}. The present paper sits before that stage. Its purpose is to isolate the intrinsic state variables from the finite-node corrected perverse/Hodge/schober architecture so that later quiver, stability, and wall-crossing formalisms have a theorem-level algebraic input.

\subsection{What this work does not yet do}

The present paper deliberately stops at the state-data layer. It does not yet define interaction or incidence data between the localized sectors, does not yet assemble a quiver-theoretic package, and does not yet introduce stability conditions, chamber structures, BPS indices, or wall-crossing laws. Those later structures require additional input not present in the intrinsic state variables alone.

More precisely, the passage from the nodewise state-data package
\[
(V_\Sigma, E_\Sigma, c_\Sigma)
\]
to interaction and incidence data requires the functorial bulk/localized coupling pattern carried by the finite-node schober package. The passage from incidence data to later dynamical structures requires, in turn, further algebraic and transport-theoretic refinements. Thus the role of Part I is intentionally limited but foundational: it isolates the finite state variables without which the later interaction, chamber, BPS, and wall-crossing constructions cannot even be formulated in intrinsic finite-node terms.

In this sense, the present paper supplies the minimal algebraic package attached to the corrected finite-node geometry. The next paper in the sequence will extract the corresponding interaction/incidence layer and thereby move from state variables to the first theorem-level quiver-theoretic structure.

\subsection{Main results}

The main results of the paper are the following.

\begin{enumerate}
\item[\textbf{(R1)}]
The corrected finite-node perverse extension determines the finite localized quotient
\[
Q_\Sigma=\bigoplus_{k=1}^r i_{k*}\Q_{\{p_k\}}
\]
and the nodewise coupling space
\[
E_\Sigma:=\Ext^1_{\Perv(X_0;\Q)}(Q_\Sigma,IC_{X_0})
\cong
\bigoplus_{k=1}^r \Q e_k.
\]

\item[\textbf{(R2)}]
The corrected global extension class determines the distinguished coefficient vector
\[
[\mathcal P]_{\mathrm{perv}}=\sum_{k=1}^r c_k e_k,
\qquad
c_\Sigma=(c_1,\dots,c_r)\in \Q^r.
\]

\item[\textbf{(R3)}]
The mixed-Hodge-module lift
\[
0\to IC^H_{X_0}\to \mathcal P^H\to \bigoplus_{k=1}^r i_{k*}\Q^H_{\{p_k\}}(-1)\to 0
\]
realizes the same finite-node state-data architecture as the corrected perverse extension.

\item[\textbf{(R4)}]
The finite-node schober datum
\[
S_\Sigma=
\Bigl(
\mathcal C_{\mathrm{bulk}},
\{\mathcal C_{p_k}\}_{k=1}^r,
\{\Phi_k,\Psi_k\}_{k=1}^r,
Sh(S_\Sigma)
\Bigr)
\]
has shadow
\[
Sh(S_\Sigma)\cong \mathcal P
\]
and is indexed by the same finite node set \(\Sigma\), thereby supplying the categorical realization of the same state variables.

\item[\textbf{(R5)}]
The three realizations above determine a symbolic dictionary at the state-data level, encoding the passage from the finite node set, localized quotient, nodewise coupling space, corrected global class, and categorical localized sectors to the common algebraic state-data package
\[
\Sigma
\rightsquigarrow
V_\Sigma,
\qquad
Q_\Sigma
\rightsquigarrow
\bigoplus_{k=1}^r \Q v_k,
\qquad
[\mathcal P]_{\mathrm{perv}}
\rightsquigarrow
c_\Sigma.
\]
\end{enumerate}

\subsection{Algebraic state-data dictionary}\label{sec:state-data-dictionary}

The purpose of this section is to record, in explicit symbolic form, the passage from the finite-node geometric and categorical data of the degeneration to the algebraic state variables extracted in the previous sections. The point is not merely notational compression. The finite-node corrected perverse extension, its mixed-Hodge-module lift, and the finite-node schober realization all determine the same node-indexed architecture, and the present dictionary isolates the algebraic objects that will serve as the intrinsic state-data layer for the later incidence, quiver, stability, and BPS analyses.

We use the symbol $\rightsquigarrow$ to denote the canonical passage from a geometric, sheaf-theoretic, mixed-Hodge, or categorical datum on the left to its associated algebraic state datum on the right. Thus $A \rightsquigarrow B$
means that $B$ is the algebraic image extracted from $A$ by the constructions proved in the preceding sections. In particular, $\rightsquigarrow$ is not asserted to be a morphism in a fixed category, nor an equivalence relation, nor a quiver arrow. It is a structural extraction symbol: it records the theorem-level assignment from finite-node input data to algebraic state data. The reason for assembling these assignments in one place is that the finite-node degeneration carries several parallel realizations of the same underlying architecture: \textit{perverse}, \textit{mixed-Hodge}, and \textit{categorical}. 

The assignments listed below should be read as canonical state-data extractions: $\text{input datum} \rightsquigarrow \text{algebraic state datum}$. The state-data extraction is summarized by the following rubric:
\[
\Sigma=\{p_1,\dots,p_r\}
\rightsquigarrow
V_\Sigma=\{v_1,\dots,v_r\},
\qquad
Q_\Sigma=\bigoplus_{k=1}^r i_{k*}\Q_{\{p_k\}}
\rightsquigarrow
\bigoplus_{k=1}^r \Q v_k,
\]
\[
E_\Sigma=
\Ext^1_{\Perv(X_0;\Q)}(Q_\Sigma,IC_{X_0})
\rightsquigarrow
\bigoplus_{k=1}^r \Q e_k,
\qquad
[\mathcal P]_{\mathrm{perv}}=\sum_{k=1}^r c_k e_k
\rightsquigarrow
c_\Sigma=(c_1,\dots,c_r)\in \Q^r.
\]
The present dictionary makes explicit how these realizations feed the same algebraic package
\[
(V_\Sigma,E_\Sigma,c_\Sigma).
\]
In this sense, the dictionary is the first algebraic rubric for the finite-node conifold degeneration: it records the state variables canonically attached to the degeneration before any later incidence, quiver, stability, or wall-crossing structures are imposed.
\subsection{Scope and organization}

The scope of the present paper is confined to the extraction of algebraic state data from the finite-node corrected perverse extension, its mixed-Hodge-module lift, and its finite-node schober shadow. The paper does not yet assemble a full quiver incidence formalism, and it does not yet construct stability conditions, BPS indices, or wall-crossing formulas. Those belong to the later papers built on top of the present state-data layer.

Section~\ref{sec:geometric-setup} fixes the finite-node geometric setting and notational conventions. Section~\ref{sec:perverse-state-data} recalls the corrected finite-node perverse extension and extracts the localized quotient and nodewise coupling space. Section~\ref{sec:coefficient-vector} isolates the coefficient vector of the corrected global extension class. Section~\ref{sec:mixed-hodge-state-data} recalls the mixed-Hodge-module lift and proves state-data compatibility under realization. Section~\ref{sec:schober-shadow-state-data} recalls the finite-node schober datum and its shadow relation to \(\mathcal P\). Section~\ref{sec:symbolic-dictionary} records the symbolic dictionary and defines the finite-node algebraic state-data package. Section~\ref{sec:compatibility-invariance} proves compatibility and invariance statements for the extracted state data.

\section{Geometric and sheaf-theoretic setup}\label{sec:geometric-setup}

We fix the geometric and sheaf-theoretic setting used throughout the paper. The local topology of the smoothing is encoded by the Milnor fibers at the singular points, while the sheaf-theoretic objects of interest live on the singular central fiber $X_0$. In particular, the nearby-cycle complex, the vanishing-cycle complex, and the corrected perverse object are objects on $X_0$, whereas the Milnor fibers supply the local topological input controlling their stalk behavior at the nodes; see, for example, \cite{MilnorSingularPoints,DimcaSheaves,BBD,KS,SaitoDuality,SaitoMHM}.

\subsection{Finite-node conifold degenerations}

We first record the precise class of degenerations considered below.

\begin{definition}[Finite-node conifold degeneration]
\label{def:finite-node-conifold-degeneration}
A \emph{finite-node conifold degeneration} consists of a proper holomorphic map
\[
\pi:\mathcal X\to\Delta
\]
from a complex analytic space $\mathcal X$ to a sufficiently small complex disk
\[
\Delta=\{t\in\C:\ |t|<\epsilon\},
\]
satisfying the following conditions:
\begin{enumerate}
\item $\mathcal X$ is a complex analytic space of pure complex dimension $4$;
\item for every sufficiently small $t\neq 0$, the fiber
\[
X_t:=\pi^{-1}(t)
\]
is a smooth complex threefold;
\item the central fiber
\[
X_0:=\pi^{-1}(0)
\]
has singular locus
\[
\Sigma=\Sing(X_0)=\{p_1,\dots,p_r\},
\]
where $\Sigma$ is a finite set;
\item each singular point $p_k\in\Sigma$ is an ordinary double point; equivalently, for each $p_k\in\Sigma$, there exists a neighborhood of $p_k$ in $\mathcal X$ and local holomorphic coordinates
\[
(x_1,x_2,x_3,x_4,t)
\]
in which $\pi$ is locally given by
\[
x_1^2+x_2^2+x_3^2+x_4^2=t;
\]
\item $X_0$ is smooth away from $\Sigma$.
\end{enumerate}
When the fibers $X_t$ for $t\neq 0$ are Calabi--Yau threefolds, we call $\pi:\mathcal X\to\Delta$ a \emph{Calabi--Yau finite-node conifold degeneration}.
\end{definition}

\begin{remark}
\label{rem:conifold-degeneration-scope}
Definition~\ref{def:finite-node-conifold-degeneration} isolates the geometric input used in the present paper. The arguments below depend on the one-parameter degeneration structure, the fact that the general fiber is a smooth complex threefold, and the assumption that the singularities of the central fiber are finitely many ordinary double points. The Calabi--Yau condition is important for the intended geometric and physical applications, but the formal sheaf-theoretic and categorical constructions developed below depend primarily on the finite-node ordinary-double-point degeneration structure itself.
\end{remark}

\subsection{Smooth locus and node inclusions}

We next fix the basic notation attached to the central fiber and its singular locus.

\begin{definition}[Smooth locus and node inclusions]
\label{def:smooth-locus-node-inclusions}
Let $\pi:\mathcal X\to\Delta$ be a finite-node conifold degeneration in the sense of Definition~\ref{def:finite-node-conifold-degeneration}. We write
\[
X_0:=\pi^{-1}(0),\qquad
\Sigma=\Sing(X_0)=\{p_1,\dots,p_r\},\qquad
U:=X_0\setminus\Sigma.
\]
Thus $U$ is the smooth locus of the central fiber $X_0$. We denote by
\[
j:U\hookrightarrow X_0
\]
the open inclusion, and for each $p_k\in\Sigma$, we denote by
\[
i_k:\{p_k\}\hookrightarrow X_0
\]
the closed inclusion of the $k$-th singular point.
\end{definition}

\begin{remark}
\label{rem:bulk-localized-geometric-split}
The decomposition $X_0=U\sqcup \Sigma$ is the geometric source of the bulk/localized distinction used throughout the paper. The bulk geometry is carried by the smooth locus $U$, while the localized contributions are concentrated at the nodes $p_k\in\Sigma$.
\end{remark}

\subsection{Local ordinary double point model}

Since all singular points of the central fiber are ordinary double points, the local model is uniform.

\begin{definition}[Local ordinary double point model]
\label{def:local-odp-model}
A \emph{local ordinary double point degeneration model} is the holomorphic map
\[
\pi_{\mathrm{loc}}:\C^4\to\C,
\qquad
\pi_{\mathrm{loc}}(x_1,x_2,x_3,x_4)=x_1^2+x_2^2+x_3^2+x_4^2.
\]
Its central fiber
\[
\pi_{\mathrm{loc}}^{-1}(0)
\]
has an isolated ordinary double point at the origin, and for every $t\neq 0$, the fiber
\[
\pi_{\mathrm{loc}}^{-1}(t)
\]
is smooth.
\end{definition}

\begin{remark}
\label{rem:local-model-role}
By Definition~\ref{def:finite-node-conifold-degeneration}, each singular point $p_k\in\Sigma$ of the central fiber is locally analytically equivalent to the model of Definition~\ref{def:local-odp-model}. This local model controls the local Milnor-fiber topology and hence the rank-one local vanishing sector at each node.
\end{remark}

\subsection{Milnor fibers and local vanishing topology}

We now record the local topological input attached to a singular point.

\begin{definition}[Local Milnor ball and local smoothing fiber]
\label{def:local-milnor-ball-fiber}
Let $\pi:\mathcal X\to\Delta$ be a finite-node conifold degeneration, and let $p_k\in\Sigma$. Choose a sufficiently small closed ball
\[
B_{\epsilon}(p_k)\subset \mathcal X
\]
centered at $p_k$ and a sufficiently small disk
\[
\Delta_\eta=\{t\in\C:\ |t|<\eta\}\subset \Delta
\]
such that the restriction of $\pi$ to
\[
B_{\epsilon}(p_k)\cap \pi^{-1}(\Delta_\eta)
\]
contains no singular point other than $p_k$. For each sufficiently small $t\in\Delta_\eta^\ast:=\Delta_\eta\setminus\{0\}$, the local smoothing fiber at $p_k$ is defined to be
\[
F_{p_k}:=X_t\cap B_{\epsilon}(p_k).
\]
\end{definition}

\begin{definition}[Milnor fiber]
\label{def:milnor-fiber}
In the setting of Definition~\ref{def:local-milnor-ball-fiber}, the space
\[
F_{p_k}=X_t\cap B_{\epsilon}(p_k)
\]
for sufficiently small $t\neq 0$ is called the \emph{Milnor fiber} of the singularity $p_k$. Its homotopy type is independent of the sufficiently small choices of $\epsilon$ and $t$; see \cite{MilnorSingularPoints,DimcaSheaves}.
\end{definition}

\begin{proposition}
\label{prop:odp-milnor-fiber-s3}
Let $p_k\in\Sigma$ be an ordinary double point of the central fiber of a finite-node conifold degeneration. Then the Milnor fiber $F_{p_k}$ has the homotopy type of $S^3$. In particular,
\[
\widetilde H^m(F_{p_k};\Q)=0
\quad\text{for }m\neq 3,
\qquad
\widetilde H^3(F_{p_k};\Q)\cong \Q.
\]
\end{proposition}

\begin{proof}
By Definition~\ref{def:finite-node-conifold-degeneration}, the singularity at $p_k$ is locally analytically equivalent to the ordinary double point model of Definition~\ref{def:local-odp-model}. The Milnor fiber of the hypersurface singularity
\[
x_1^2+x_2^2+x_3^2+x_4^2=t
\]
has the homotopy type of $S^3$ by the classical theory of isolated hypersurface singularities; see \cite[Chapter~5]{MilnorSingularPoints} and \cite[Section~4.1]{DimcaSheaves}. The cohomological statement follows immediately.
\end{proof}

\begin{remark}
\label{rem:one-local-sector-per-node}
Proposition~\ref{prop:odp-milnor-fiber-s3} is the local topological source of the one-node/one-sector principle used throughout the paper. Each ordinary double point contributes one rank-one middle-dimensional vanishing sector.
\end{remark}

\subsection{Nearby cycles, vanishing cycles, and the variation morphism}

We now pass from the local topology of the Milnor fiber to the sheaf-theoretic objects on the singular central fiber.

\begin{definition}[Nearby and vanishing cycles]
\label{def:nearby-vanishing-cycles}
Let $\pi:\mathcal X\to\Delta$ be a finite-node conifold degeneration, and let
\[
F:=\Q_{\mathcal X}[3].
\]
We denote by
\[
\psi_\pi(F)
\qquad\text{and}\qquad
\phi_\pi(F)
\]
the nearby-cycle and vanishing-cycle complexes of $F$ on the central fiber $X_0$; see \cite{BBD,KS,DimcaSheaves,SaitoDuality,SaitoMHM}.
\end{definition}

\begin{remark}
\label{rem:cycles-live-on-central-fiber}
The objects $\psi_\pi(F)$ and $\phi_\pi(F)$ live on the singular central fiber $X_0$, not on the Milnor fiber. Their local stalks at the singular points are controlled by the topology of the corresponding Milnor fibers. In particular, the Milnor fiber is the local topological model behind the stalks of the vanishing-cycle complex, not the space on which the final corrected perverse object is defined.
\end{remark}

\begin{proposition}
\label{prop:vanishing-stalks-controlled-by-milnor}
Let $p_k\in\Sigma$. Then the local stalk cohomology of the vanishing-cycle complex $\phi_\pi(F)$ at $p_k$ is controlled by the reduced cohomology of the Milnor fiber $F_{p_k}$. In the ordinary double point case, this local vanishing contribution is rank one and concentrated in middle degree.
\end{proposition}

\begin{proof}
By the standard nearby-/vanishing-cycle formalism for isolated hypersurface singularities, the local stalk cohomology of $\phi_\pi(F)$ at $p_k$ is computed by the reduced cohomology of the Milnor fiber $F_{p_k}$; see \cite[Section~5.1]{DimcaSheaves}, \cite[Expos\'e XIII]{BBD}, and \cite{SaitoDuality}. The ordinary double point case therefore reduces to Proposition~\ref{prop:odp-milnor-fiber-s3}, which shows that the reduced cohomology is rank one in degree $3$ and vanishes in all other degrees.
\end{proof}

\begin{definition}[Variation morphism]
\label{def:variation-morphism}
In the nearby-/vanishing-cycle formalism there is a natural variation morphism
\[
\var_F:\phi_\pi(F)\to \psi_\pi(F).
\]
\end{definition}

\subsection{The corrected finite-node perverse object}

We now record the corrected perverse object attached to the degeneration.

\begin{definition}[Corrected finite-node perverse object]
\label{def:corrected-finite-node-perverse-object}
Let $\pi:\mathcal X\to\Delta$ be a finite-node conifold degeneration, let
\[
F:=\Q_{\mathcal X}[3],
\]
and let
\[
\var_F:\phi_\pi(F)\to\psi_\pi(F)
\]
be the variation morphism of Definition~\ref{def:variation-morphism}. The \emph{corrected finite-node perverse object} associated with the degeneration is defined by
\[
\mathcal P:=\Cone(\var_F)[-1].
\]
\end{definition}

\begin{proposition}
\label{prop:corrected-object-is-perverse}
The object
\[
\mathcal P:=\Cone(\var_F)[-1]
\]
is a perverse sheaf on the singular central fiber $X_0$.
\end{proposition}

\begin{proof}
The nearby-cycle and vanishing-cycle constructions, with the standard shift conventions in the present setting, produce perverse objects on $X_0$; see \cite[Expos\'e XIII]{BBD}, \cite[Chapter~10]{KS}, \cite{SaitoDuality}, and \cite{SaitoMHM}. Since $\var_F:\phi_\pi(F)\to\psi_\pi(F)$ is a morphism in the perverse category, the shifted cone $\Cone(\var_F)[-1]$ is again perverse.
\end{proof}

\begin{remark}
\label{rem:corrected-object-not-on-milnor-fiber}
The corrected perverse object $\mathcal P=\Cone(\var_F)[-1]$ is not a perverse sheaf on the Milnor fiber. It is a perverse sheaf on the singular central fiber $X_0$, constructed from nearby and vanishing cycles on $X_0$ whose local content is computed by the Milnor fibers.
\end{remark}

\begin{remark}
\label{rem:point-supported-correction-from-milnor}
The rank-one reduced cohomology of the Milnor fiber at each node is the local topological source of the point-supported correction term that appears later in the corrected finite-node perverse extension. Thus the local Milnor-fiber topology is the seed from which the global corrected object and its localized sectors are built.
\end{remark}

\section{Corrected perverse state data}\label{sec:perverse-state-data}

The purpose of this section is to isolate the perverse-side state data canonically attached to a finite-node conifold degeneration. The input is the corrected finite-node perverse object
\[
\mathcal P\in \Perv(X_0;\Q)
\]
constructed from nearby and vanishing cycles in the earlier papers. The output is the finite localized quotient
\[
Q_\Sigma:=\bigoplus_{k=1}^r i_{k*}\Q_{\{p_k\}},
\]
the corresponding nodewise coupling space
\[
E_\Sigma:=\Ext^1_{\Perv(X_0;\Q)}(Q_\Sigma,IC_{X_0}),
\]
and its distinguished nodewise basis
\[
E_\Sigma\cong \bigoplus_{k=1}^r \Q e_k.
\]
These are the perverse state variables underlying the later coefficient, incidence, and quiver-theoretic constructions.

\subsection{Finite-node corrected extension}

We begin by recalling the corrected finite-node perverse extension. Let
\[
\pi:X\to \Delta
\]
be a one-parameter degeneration as in Section~\ref{sec:geometric-setup}, and let
\[
\Sigma:=\{p_1,\dots,p_r\}\subset X_0
\]
be the finite node set of the central fiber. Write
\[
i_k:\{p_k\}\hookrightarrow X_0
\qquad (1\le k\le r)
\]
for the closed immersions of the nodes. The corrected finite-node perverse object \(\mathcal P\) is the canonical perverse sheaf obtained from the nearby-/vanishing-cycle formalism in the multi-node setting; see \cite{RahmanPerverseNearbyCycles,RahmanMixedHodgeModules}. The basic structural fact needed here is the following.

\begin{theorem}[finite-node corrected extension]\label{thm:finite-node-corrected-extension}
Let \(\pi:X\to \Delta\) be a finite-node conifold degeneration. Then the corrected perverse object
\[
\mathcal P\in \Perv(X_0;\Q)
\]
fits into a short exact sequence
\begin{equation}\label{eq:finite-node-corrected-extension}
0\longrightarrow IC_{X_0}\longrightarrow \mathcal P\longrightarrow \bigoplus_{k=1}^r i_{k*}\Q_{\{p_k\}}\longrightarrow 0.
\end{equation}
Equivalently, \(\mathcal P\) is an extension of the bulk term \(IC_{X_0}\) by the finite direct sum of the point-supported rank-one node contributions.
\end{theorem}

\begin{proof}
This is the finite-node extension theorem established on the perverse side in \cite[Theorem~1.2]{RahmanPerverseNearbyCycles} and realized internally through the mixed-Hodge-module refinement in \cite[Theorem~1.3]{RahmanMixedHodgeModules}. The same exact sequence is recovered as the shadow of the finite-node schober datum in \cite[Theorem~5.2]{RahmanMultiNodeSchoberPaper}. We use \eqref{eq:finite-node-corrected-extension} here as recalled input.
\end{proof}

\subsection{Localized quotient and node indexing}

The quotient term in \eqref{eq:finite-node-corrected-extension} is the first finite state variable attached to the degeneration.

\begin{definition}[finite localized quotient]\label{def:finite-localized-quotient}
Define
\begin{equation}\label{eq:def-QSigma}
Q_\Sigma:=\bigoplus_{k=1}^r i_{k*}\Q_{\{p_k\}}
\in \Perv(X_0;\Q).
\end{equation}
We call \(Q_\Sigma\) the \emph{finite localized quotient} of the degeneration.
\end{definition}

\begin{remark}\label{rem:QSigma-node-indexing}
The object \(Q_\Sigma\) is supported on the finite set \(\Sigma\) and contains exactly one point-supported rank-one summand for each node \(p_k\in\Sigma\). In particular, the node set \(\Sigma\) canonically indexes the localized singular terms in the corrected extension \eqref{eq:finite-node-corrected-extension}. This is the perverse-side origin of the finite vertex set introduced later in the algebraic dictionary.
\end{remark}

The finite localized quotient is compatible with both the mixed-Hodge and schober realizations. On the mixed-Hodge side, \(\mathcal P\) is lifted to an object \(\mathcal P^H\in MHM(X_0)\) whose quotient is
\[
Q_\Sigma^H:=\bigoplus_{k=1}^r i_{k*}\Q^H_{\{p_k\}}(-1),
\]
with \(\rat(Q_\Sigma^H)\cong Q_\Sigma\); see \cite[Section~5.6 and Theorem~1.3]{RahmanMixedHodgeModules}. On the categorical side, the finite-node schober datum has one localized categorical sector \(\mathcal C_{p_k}\) for each node \(p_k\in\Sigma\), and its shadow recovers \(\mathcal P\); see \cite[Theorem~5.2 and Section~5.4]{RahmanMultiNodeSchoberPaper}. Thus the same finite node set indexes the localized quotient on the perverse, mixed-Hodge, and categorical levels.

\subsection{Nodewise coupling space}

The second perverse state variable is the global extension space determined by eqn. \eqref{eq:finite-node-corrected-extension}. Since the quotient \(Q_\Sigma\) is a finite direct sum indexed by the node set, the relevant extension space decomposes nodewise.

\begin{definition}[nodewise coupling space]\label{def:nodewise-coupling-space}
Define
\begin{equation}\label{eq:def-ESigma}
E_\Sigma:=\Ext^1_{\Perv(X_0;\Q)}(Q_\Sigma,IC_{X_0}).
\end{equation}
We call \(E_\Sigma\) the \emph{nodewise coupling space} of the finite-node degeneration.
\end{definition}

The formal additivity of \(\Ext^1\) over the finite direct sum \(Q_\Sigma\) yields the first decomposition theorem.

\begin{lemma}[finite additivity of the perverse extension space]\label{lem:finite-additivity-perverse-ext}
There is a natural isomorphism
\begin{equation}\label{eq:ESigma-direct-sum}
E_\Sigma
:=
\Ext^1_{\Perv(X_0;\Q)}(Q_\Sigma,IC_{X_0})
\cong
\bigoplus_{k=1}^r
\Ext^1_{\Perv(X_0;\Q)}\!\bigl(i_{k*}\Q_{\{p_k\}},IC_{X_0}\bigr).
\end{equation}
\end{lemma}

\begin{proof}
This is exactly the finite-additivity statement proved in \cite[Lemma~5.4]{RahmanMixedHodgeModules}. Since
\[
Q_\Sigma=\bigoplus_{k=1}^r i_{k*}\Q_{\{p_k\}},
\]
the functor \(\Hom_{\Perv(X_0;\Q)}(-,IC_{X_0})\) sends the finite direct sum in the first variable to a finite direct product, and for a finite index set this agrees with the direct sum. Passing to the first right-derived functor gives \eqref{eq:ESigma-direct-sum}.
\end{proof}

The decisive refinement is that each summand in \eqref{eq:ESigma-direct-sum} is one-dimensional and carries a distinguished generator.

\begin{theorem}[strong nodewise coupling theorem]\label{thm:strong-nodewise-coupling}
For each node \(p_k\in\Sigma\), one has
\begin{equation}\label{eq:local-channel-dimension}
\dim_\Q
\Ext^1_{\Perv(X_0;\Q)}\!\bigl(i_{k*}\Q_{\{p_k\}},IC_{X_0}\bigr)=1.
\end{equation}
Moreover, the local corrected perverse ODP extension defines a distinguished generator
\begin{equation}\label{eq:def-ek}
e_k\in \Ext^1_{\Perv(X_0;\Q)}\!\bigl(i_{k*}\Q_{\{p_k\}},IC_{X_0}\bigr).
\end{equation}
Consequently,
\begin{equation}\label{eq:ESigma-basis}
E_\Sigma
\cong
\bigoplus_{k=1}^r \Q e_k.
\end{equation}
\end{theorem}

\begin{proof}
The one-dimensionality statement and existence of the distinguished generators \(e_k\) are proved in \cite[Corollary~5.12 and Definition~5.13]{RahmanMixedHodgeModules}. Substituting these generators into the decomposition \eqref{eq:ESigma-direct-sum} yields \eqref{eq:ESigma-basis}; see \cite[Corollary~5.14]{RahmanMixedHodgeModules}.
\end{proof}

\begin{definition}[nodewise coupling lines]\label{def:nodewise-coupling-lines}
For each \(1\le k\le r\), the one-dimensional subspace
\[
\Q e_k\subset E_\Sigma
\]
is called the \emph{nodewise coupling line} associated with the node \(p_k\).
\end{definition}

\begin{remark}\label{rem:nodewise-coupling-lines}
Theorem~\ref{thm:strong-nodewise-coupling} identifies the first explicit algebraic shadow of the finite-node degeneration: each node contributes a distinguished one-dimensional coupling channel, and the full perverse extension space is their direct sum. Thus the finite node set \(\Sigma\) does not merely label singular points geometrically; it labels the canonical basis of local coupling directions in the corrected global extension package.
\end{remark}
We therefore arrive at the perverse state-data package attached to the degeneration.

\begin{definition}[perverse state-data package]\label{def:perverse-state-data-package}
The \emph{perverse state-data package} of the finite-node conifold degeneration is the triple
\begin{equation}\label{eq:def-perverse-state-data-package}
(Q_\Sigma,E_\Sigma,\{e_1,\dots,e_r\}),
\end{equation}
where \(Q_\Sigma\) is the finite localized quotient of Definition~\ref{def:finite-localized-quotient}, \(E_\Sigma\) is the nodewise coupling space of Definition~\ref{def:nodewise-coupling-space}, and \(\{e_1,\dots,e_r\}\) is the distinguished nodewise basis of Theorem~\ref{thm:strong-nodewise-coupling}.
\end{definition}

The next section extracts from the global corrected extension class
\[
[\mathcal P]_{\mathrm{perv}}\in E_\Sigma
\]
its unique coordinate vector in the distinguished basis \(\{e_1,\dots,e_r\}\). This produces the coefficient vector
\[
c_\Sigma=(c_1,\dots,c_r)\in \Q^r,
\]
which completes the finite algebraic state-data layer.

\section{Coefficient vector}\label{sec:coefficient-vector}

The preceding section isolates the finite localized quotient
\[
Q_\Sigma=\bigoplus_{k=1}^r i_{k*}\Q_{\{p_k\}}
\]
and the nodewise coupling space
\[
E_\Sigma:=\Ext^1_{\Perv(X_0;\Q)}(Q_\Sigma,IC_{X_0})
\cong \bigoplus_{k=1}^r \Q e_k.
\]
The purpose of the present section is to extract from the corrected global perverse extension class its coordinates relative to the distinguished nodewise basis $\{e_1,\dots,e_r\}$. These coordinates form the finite coefficient vector attached to the degeneration. This coefficient vector is the final perverse-side state variable needed for the algebraic state-data package.

\subsection{Expansion of the corrected global class}

We begin with the corrected global extension class determined by the finite-node corrected perverse extension
\[
0\to IC_{X_0}\to \mathcal P\to Q_\Sigma\to 0
\]
of Theorem~\ref{thm:finite-node-corrected-extension}. By definition, this exact sequence determines an element
\[
[\mathcal P]_{\mathrm{perv}}\in \Ext^1_{\Perv(X_0;\Q)}(Q_\Sigma,IC_{X_0}):=E_\Sigma.
\]

Since Section~\ref{sec:perverse-state-data} identified $E_\Sigma$ with the direct sum $\bigoplus_{k=1}^r \Q e_k$, the global corrected class admits a unique nodewise expansion.

\begin{proposition}[expansion of the corrected global class]
\label{prop:global-class-expansion}
The corrected global perverse extension class admits a unique expansion
\begin{equation}\label{eq:global-class-expansion}
[\mathcal P]_{\mathrm{perv}}:=\sum_{k=1}^r c_k e_k
\end{equation}
for uniquely determined coefficients $c_k\in \Q$.
\end{proposition}

\begin{proof}
By Theorem~\ref{thm:strong-nodewise-coupling}, the extension space $E_\Sigma$ has distinguished basis $\{e_1,\dots,e_r\}$ and satisfies
\[
E_\Sigma\cong \bigoplus_{k=1}^r \Q e_k.
\]
Therefore every element of $E_\Sigma$, and in particular the corrected global extension class $[\mathcal P]_{\mathrm{perv}}$, admits a unique expansion in that basis. This is exactly the statement proved in \cite[Proposition~5.15]{RahmanMixedHodgeModules}.
\end{proof}

\begin{remark}
\label{rem:nodewise-coordinates-global-class}
Equation~\eqref{eq:global-class-expansion} is the first point at which the corrected global extension becomes explicitly finite-dimensional. The global corrected class is no longer treated only as an abstract extension class; it is now represented by a concrete tuple of coefficients relative to the distinguished nodewise coupling lines.
\end{remark}

\subsection{Definition of the coefficient vector}

The coefficients appearing in Proposition~\ref{prop:global-class-expansion} are assembled into a single finite vector.

\begin{definition}[coefficient vector]
\label{def:coefficient-vector}
Let
\[
[\mathcal P]_{\mathrm{perv}}=\sum_{k=1}^r c_k e_k
\]
be the unique expansion of Proposition~\ref{prop:global-class-expansion}. The \emph{coefficient vector} of the finite-node conifold degeneration is the vector
\begin{equation}\label{eq:def-coefficient-vector}
c_\Sigma:=(c_1,\dots,c_r)\in \Q^r.
\end{equation}
\end{definition}

\begin{remark}
\label{rem:coefficient-vector-node-indexed}
The coefficient vector $c_\Sigma$ is indexed by the node set $\Sigma=\{p_1,\dots,p_r\}$. Its $k$-th coordinate records the component of the corrected global extension class along the distinguished nodewise coupling line $\Q e_k\subset E_\Sigma$. Thus the node set does not merely index singular points geometrically; it also indexes the canonical coordinates of the corrected global extension package.
\end{remark}

The coefficient vector is the third component of the state-data package attached to the degeneration:
\[
(Q_\Sigma,E_\Sigma,\{e_1,\dots,e_r\})
\quad\leadsto\quad
(Q_\Sigma,E_\Sigma,c_\Sigma).
\]
It will later be incorporated into the algebraic state-data dictionary as the finite coordinate package determined by the corrected global class.

\subsection{Intrinsicity of the nodewise coefficients}

The coefficients $c_1,\dots,c_r$ are not auxiliary choices. Their intrinsicity follows from two facts: first, the corrected global extension class is canonically determined by the corrected finite-node perverse extension; second, the basis $\{e_1,\dots,e_r\}$ is canonically indexed by the node set through the local corrected perverse ODP extensions; see \cite[Definition~5.13 and Corollary~5.14]{RahmanMixedHodgeModules}.

\begin{proposition}[intrinsicity of the coefficient vector]
\label{prop:intrinsicity-coefficient-vector}
The coefficient vector $c_\Sigma=(c_1,\dots,c_r)\in \Q^r$ depends only on the corrected finite-node perverse extension package
\[
0\to IC_{X_0}\to \mathcal P\to Q_\Sigma\to 0
\]
together with the distinguished nodewise basis $\{e_1,\dots,e_r\}$ of $E_\Sigma$. Equivalently, once the finite-node corrected extension package has been fixed, the coefficients $c_k$ are uniquely determined.
\end{proposition}

\begin{proof}
The corrected finite-node perverse extension determines a canonical class
\[
[\mathcal P]_{\mathrm{perv}}\in E_\Sigma.
\]
By Proposition~\ref{prop:global-class-expansion}, this class has a unique expansion in the distinguished basis $\{e_1,\dots,e_r\}$. Hence the coefficients $c_k$ are uniquely determined by the corrected extension class together with the fixed distinguished nodewise basis. No further choices enter. This is the content of \cite[Proposition~5.15]{RahmanMixedHodgeModules}, read together with the construction of the generators in \cite[Definition~5.13]{RahmanMixedHodgeModules}.
\end{proof}

\begin{remark}
\label{rem:coefficient-vector-first-global-state-variable}
The coefficient vector is the first explicit global coordinate package attached to the finite-node degeneration. The localized quotient $Q_\Sigma$ records the finite singular support, and the coupling space $E_\Sigma$ records the nodewise extension channels. The vector $c_\Sigma$ records how the corrected global extension class is assembled from those channels. In this sense, $c_\Sigma$ is the first genuine global state variable extracted from the degeneration.
\end{remark}

We may therefore enlarge the perverse state-data package of Definition~\ref{def:perverse-state-data-package} to the coordinate-level package
\begin{equation}\label{eq:def-perverse-coordinate-package}
(Q_\Sigma,E_\Sigma,\{e_1,\dots,e_r\},c_\Sigma).
\end{equation}
The next section shows that the same finite-node state-data architecture is realized on the mixed-Hodge side.

\section{Mixed-Hodge state data}\label{sec:mixed-hodge-state-data}

The purpose of this section is to place the finite-node state-data package extracted on the perverse side into its mixed-Hodge-module refinement. The key point is that the corrected finite-node perverse extension is not only a perverse-sheaf object on $X_0$, but also admits a canonical lift to the category $MHM(X_0)$. This lift preserves the same finite bulk/localized architecture: the bulk term is the intersection-complex Hodge module, the localized quotient is the finite direct sum of point-supported Tate-twisted Hodge modules at the nodes, and the realization functor recovers the corrected perverse extension. Thus the mixed-Hodge-module layer does not introduce a different finite-node structure; it refines the same one Hodge-theoretically; see \cite{RahmanMixedHodgeModules}.

\subsection{Finite-node mixed-Hodge-module lift}

We begin by recalling the mixed-Hodge-module refinement of the corrected finite-node perverse extension.

\begin{theorem}[finite-node mixed-Hodge-module lift]
\label{thm:finite-node-mhm-lift}
Let $\pi:\mathcal X\to\Delta$ be a finite-node conifold degeneration. Then there exists an object
\[
\mathcal P^H\in MHM(X_0)
\]
fitting into a short exact sequence
\begin{equation}\label{eq:finite-node-mhm-extension}
0\longrightarrow IC^H_{X_0}
\longrightarrow \mathcal P^H
\longrightarrow Q_\Sigma^H
\longrightarrow 0,
\qquad
Q_\Sigma^H:=\bigoplus_{k=1}^r i_{k*}\Q^H_{\{p_k\}}(-1).
\end{equation}
\end{theorem}

\begin{proof}
This is the finite-node mixed-Hodge-module extension theorem established in \cite[Theorem~1.3]{RahmanMixedHodgeModules}. The bulk term is the intersection-complex Hodge module $IC^H_{X_0}$, and the singular quotient is the finite direct sum $Q_\Sigma^H=\bigoplus_{k=1}^r i_{k*}\Q^H_{\{p_k\}}(-1)$ of the point-supported rank-one localized Hodge blocks, one for each node.
\end{proof}

\begin{definition}[finite mixed-Hodge localized quotient]
\label{def:finite-hodge-localized-quotient}
We write
\begin{equation}\label{eq:def-QSigmaH}
Q_\Sigma^H:=\bigoplus_{k=1}^r i_{k*}\Q^H_{\{p_k\}}(-1)\in MHM(X_0)
\end{equation}
for the localized quotient in the mixed-Hodge-module extension \eqref{eq:finite-node-mhm-extension}. We call $Q_\Sigma^H$ the \emph{finite mixed-Hodge localized quotient}.
\end{definition}

\begin{remark}
\label{rem:hodge-localized-quotient-node-indexed}
The object $Q_\Sigma^H$ contains one point-supported Tate-twisted Hodge summand for each node $p_k\in\Sigma$. Thus the same finite node set $\Sigma=\{p_1,\dots,p_r\}$ that indexed the localized quotient
\[
Q_\Sigma=\bigoplus_{k=1}^r i_{k*}\Q_{\{p_k\}}
\]
on the perverse side also indexes the localized quotient on the mixed-Hodge side.
\end{remark}

The mixed-Hodge-module extension also carries a nodewise extension-theoretic decomposition analogous to the perverse-side decomposition of Section~\ref{sec:perverse-state-data}.

\begin{lemma}[finite additivity on the mixed-Hodge side]
\label{lem:hodge-finite-additivity}
There is a natural isomorphism
\begin{equation}\label{eq:hodge-ext-direct-sum}
\Ext^1_{MHM(X_0)}(Q_\Sigma^H,IC^H_{X_0})
\cong
\bigoplus_{k=1}^r
\Ext^1_{MHM(X_0)}\bigl(i_{k*}\Q^H_{\{p_k\}}(-1),IC^H_{X_0}\bigr).
\end{equation}
\end{lemma}

\begin{proof}
This is the finite-additivity statement proved in \cite[Lemma~5.5]{RahmanMixedHodgeModules}. Since
\[
Q_\Sigma^H=\bigoplus_{k=1}^r i_{k*}\Q^H_{\{p_k\}}(-1),
\]
the same direct-sum argument used on the perverse side applies in the abelian category $MHM(X_0)$.
\end{proof}

\begin{proposition}[global mixed-Hodge extension class]
\label{prop:global-hodge-extension-class}
The exact sequence \eqref{eq:finite-node-mhm-extension} determines a canonical global class
\[
[\mathcal P^H]\in \Ext^1_{MHM(X_0)}(Q_\Sigma^H,IC^H_{X_0}),
\]
and under the decomposition \eqref{eq:hodge-ext-direct-sum} this class determines, and is equivalently determined by, a tuple of nodewise extension classes
\[
(\epsilon_1^H,\dots,\epsilon_r^H),
\qquad
\epsilon_k^H\in
\Ext^1_{MHM(X_0)}\bigl(i_{k*}\Q^H_{\{p_k\}}(-1),IC^H_{X_0}\bigr).
\]
\end{proposition}

\begin{proof}
This is exactly \cite[Proposition~5.6]{RahmanMixedHodgeModules}. The global corrected extension class in $MHM(X_0)$ is therefore assembled from localized node-to-bulk Hodge-theoretic extension data in the same sense that the perverse corrected class is assembled from localized nodewise perverse extension channels.
\end{proof}

\subsection{Realization compatibility}

We now compare the mixed-Hodge-module lift with the corrected perverse state data of Sections~\ref{sec:perverse-state-data} and \ref{sec:coefficient-vector}.

\begin{theorem}[realization compatibility]
\label{thm:realization-compatibility}
The realization functor
\[
\rat:MHM(X_0)\to \Perv(X_0;\Q)
\]
sends the finite-node mixed-Hodge-module extension \eqref{eq:finite-node-mhm-extension} to the corrected finite-node perverse extension
\begin{equation}\label{eq:realized-perverse-extension}
0\longrightarrow IC_{X_0}
\longrightarrow \mathcal P
\longrightarrow Q_\Sigma
\longrightarrow 0,
\qquad
Q_\Sigma=\bigoplus_{k=1}^r i_{k*}\Q_{\{p_k\}}.
\end{equation}
In particular,
\begin{equation}\label{eq:rat-PH-is-P}
\rat(\mathcal P^H)\cong \mathcal P
\end{equation}
and
\begin{equation}\label{eq:rat-QSigmaH-is-QSigma}
\rat(Q_\Sigma^H)\cong Q_\Sigma.
\end{equation}
\end{theorem}

\begin{proof}
The realization statement for $\mathcal P^H$ is part of the main theorem package of \cite[Theorem~1.3]{RahmanMixedHodgeModules}. The identification
\[
\rat(Q_\Sigma^H)\cong Q_\Sigma
\]
is recorded explicitly in \cite[Section~5.6]{RahmanMixedHodgeModules}, where applying the realization functor to the point-supported quotient object gives
\[
\rat\!\left(\bigoplus_{k=1}^r i_{k*}\Q^H_{\{p_k\}}(-1)\right)
\cong
\bigoplus_{k=1}^r i_{k*}\Q_{\{p_k\}}.
\]
Exactness of realization then recovers the corrected finite-node perverse extension.
\end{proof}

\begin{remark}
\label{rem:mhm-refinement-not-replacement}
The mixed-Hodge-module lift does not replace the corrected perverse extension. Rather, it internalizes the same extension-theoretic structure in Saito's category. On the perverse side one sees the corrected extension
\[
0\to IC_{X_0}\to \mathcal P\to Q_\Sigma\to 0;
\]
on the mixed-Hodge side one sees the Hodge-theoretic refinement
\[
0\to IC^H_{X_0}\to \mathcal P^H\to Q_\Sigma^H\to 0.
\]
The passage from $\mathcal P^H$ to $\mathcal P$ is controlled by realization.
\end{remark}

\subsection{State-data compatibility}

We now record the form in which the mixed-Hodge-module refinement is compatible with the state-data package extracted on the perverse side.

\begin{proposition}[same node set]
\label{prop:same-node-set-hodge}
The finite node set
\[
\Sigma=\{p_1,\dots,p_r\}
\]
indexes the localized quotient on both levels:
\[
Q_\Sigma=\bigoplus_{k=1}^r i_{k*}\Q_{\{p_k\}}
\qquad\text{and}\qquad
Q_\Sigma^H=\bigoplus_{k=1}^r i_{k*}\Q^H_{\{p_k\}}(-1).
\]
\end{proposition}

\begin{proof}
This is immediate from Definitions~\ref{def:finite-localized-quotient} and \ref{def:finite-hodge-localized-quotient}. The same finite node set labels the localized point-supported summands in both categories.
\end{proof}

\begin{proposition}[same finite localized architecture]
\label{prop:same-localized-architecture}
The mixed-Hodge-module extension \eqref{eq:finite-node-mhm-extension} and the corrected perverse extension \eqref{eq:realized-perverse-extension} have the same finite bulk/localized architecture: one global bulk object and one localized rank-one summand at each node.
\end{proposition}

\begin{proof}
On the perverse side, the bulk term is $IC_{X_0}$ and the localized quotient is $Q_\Sigma=\bigoplus_{k=1}^r i_{k*}\Q_{\{p_k\}}$. On the mixed-Hodge side, the bulk term is $IC^H_{X_0}$ and the localized quotient is $Q_\Sigma^H=\bigoplus_{k=1}^r i_{k*}\Q^H_{\{p_k\}}(-1)$. By Theorem~\ref{thm:realization-compatibility}, realization identifies the second quotient with the first. Thus the same finite node-indexed architecture appears on both levels; see \cite[Section~5.6]{RahmanMixedHodgeModules}.
\end{proof}

\begin{proposition}[state-data shadow]
\label{prop:state-data-shadow}
The perverse state-data package
\[
(Q_\Sigma,E_\Sigma,\{e_1,\dots,e_r\},c_\Sigma)
\]
is the realized algebraic shadow of the mixed-Hodge-module extension package
\[
(IC^H_{X_0},Q_\Sigma^H,[\mathcal P^H]).
\]
\end{proposition}

\begin{proof}
By Theorem~\ref{thm:realization-compatibility}, realization sends the mixed-Hodge-module extension \eqref{eq:finite-node-mhm-extension} to the corrected perverse extension \eqref{eq:realized-perverse-extension}. By Section~\ref{sec:perverse-state-data}, the latter determines the nodewise coupling space
\[
E_\Sigma=\Ext^1_{\Perv(X_0;\Q)}(Q_\Sigma,IC_{X_0})
\cong \bigoplus_{k=1}^r \Q e_k,
\]
and by Section~\ref{sec:coefficient-vector}, the corrected global perverse class determines the coefficient vector
\[
c_\Sigma=(c_1,\dots,c_r)\in\Q^r.
\]
Thus the finite-node mixed-Hodge-module package refines, and after realization recovers, exactly the perverse-side state-data package. This is the compatibility principle emphasized in \cite[Section~5.6 and Section~8.3]{RahmanMixedHodgeModules}.
\end{proof}

\begin{remark}
\label{rem:hodge-state-data-role}
At the level of the present paper, the mixed-Hodge-module refinement contributes two things. First, it shows that the finite-node state variables extracted on the perverse side are not accidental artifacts of the abelian category $\Perv(X_0;\Q)$, but arise from an internal Hodge-theoretic lift. Second, it identifies the same finite localized quotient as the Hodge-theoretic refinement of the local vanishing contribution. In this sense, the mixed-Hodge-module layer confirms that the state-data package is intrinsic to the finite-node degeneration itself.
\end{remark}

The next section recalls the finite-node schober datum and shows that the same finite-node architecture also appears categorically, with one localized sector for each node and shadow equal to the corrected perverse object.

\section{Schober shadow and categorical state data}\label{sec:schober-shadow-state-data}

The purpose of this section is to record the categorical realization of the same finite-node state-data architecture already identified on the perverse and mixed-Hodge sides. The key point is that the finite-node conifold degeneration determines a schober datum with one localized categorical sector for each node, a common bulk category attached to the smooth locus, and attachment functors coupling the localized sectors to that bulk category. The resulting datum is not external to the corrected extension picture: its decategorified shadow recovers the corrected finite-node perverse object $\mathcal P$. Thus the finite-node architecture appears simultaneously in perverse, mixed-Hodge, and categorical form; see \cite{RahmanMultiNodeSchoberPaper}.

\subsection{Finite-node schober datum}

We begin by recalling the global categorical object attached to the finite-node degeneration.

\begin{theorem}[finite-node schober datum]
\label{thm:finite-node-schober-datum}
Let $\pi:\mathcal X\to\Delta$ be a finite-node conifold degeneration with node set $\Sigma=\{p_1,\dots,p_r\}$. Then there exists a finite-node schober datum
\begin{equation}\label{eq:finite-node-schober-datum}
S_\Sigma=
\Bigl(
\mathcal C_{\mathrm{bulk}},
\{\mathcal C_{p_k}\}_{k=1}^r,
\{\Phi_k,\Psi_k\}_{k=1}^r,
Sh(S_\Sigma)
\Bigr),
\end{equation}
where $\mathcal C_{\mathrm{bulk}}$ is the bulk category attached to the smooth sector $U=X_0\setminus\Sigma$, each $\mathcal C_{p_k}$ is a localized categorical sector attached to the node $p_k$, and
\[
\Phi_k:\mathcal C_{p_k}\to \mathcal C_{\mathrm{bulk}},
\qquad
\Psi_k:\mathcal C_{\mathrm{bulk}}\to \mathcal C_{p_k}
\]
are the attachment functors.
\end{theorem}

\begin{proof}
This is the global assembly theorem for the finite-node schober datum proved in \cite[Theorem~4.4]{RahmanMultiNodeSchoberPaper}. The local ODP schober blocks are constructed nodewise and then assembled into a global finite-node schober datum indexed by the node set $\Sigma$.
\end{proof}

\begin{definition}[finite-node schober package]
\label{def:finite-node-schober-package}
The object $S_\Sigma$ of \eqref{eq:finite-node-schober-datum} is called the \emph{finite-node schober package} attached to the degeneration.
\end{definition}

\begin{remark}
\label{rem:schober-package-ingredients}
The datum $S_\Sigma$ contains three kinds of categorical information: the common bulk category $\mathcal C_{\mathrm{bulk}}$, the finite family of localized sectors $\mathcal C_{p_1},\dots,\mathcal C_{p_r}$, and the attachment functors $\Phi_k,\Psi_k$ coupling each node sector to the bulk category. At the level of the present paper, these data are used only to isolate the categorical shadow of the same finite-node state-data architecture already visible in the corrected perverse extension and its mixed-Hodge lift.
\end{remark}

\subsection{Localized categorical sectors}

The first categorical state variable is the finite family of localized sectors indexed by the node set.

\begin{theorem}[localized sector theorem]
\label{thm:localized-sector-theorem}
Let
\[
S_\Sigma=
\Bigl(
\mathcal C_{\mathrm{bulk}},
\{\mathcal C_{p_k}\}_{k=1}^r,
\{\Phi_k,\Psi_k\}_{k=1}^r,
Sh(S_\Sigma)
\Bigr)
\]
be the finite-node schober datum of Theorem~\ref{thm:finite-node-schober-datum}. Then for each node $p_k\in\Sigma$ there is a distinguished localized categorical sector $\mathcal C_{p_k}$ attached to $p_k$, and $S_\Sigma$ contains exactly one such localized sector for each node of the degeneration.
\end{theorem}

\begin{proof}
This is exactly \cite[Theorem~6.1]{RahmanMultiNodeSchoberPaper}. The local ODP schober datum at each node contributes one localized category, and the global assembly preserves these categories separately because the node neighborhoods are chosen pairwise disjoint. Thus no two distinct node sectors are identified, and the family $\{\mathcal C_{p_k}\}_{k=1}^r$ is canonically indexed by the finite node set $\Sigma$.
\end{proof}

\begin{definition}[localized categorical family]
\label{def:localized-categorical-family}
We write
\begin{equation}\label{eq:def-localized-categorical-family}
\mathcal C_\Sigma:=\{\mathcal C_{p_k}\}_{k=1}^r
\end{equation}
for the finite family of localized categorical sectors attached to the nodes of the degeneration.
\end{definition}

\begin{remark}
\label{rem:one-localized-sector-per-node}
Theorem~\ref{thm:localized-sector-theorem} is the categorical analogue of the perverse-side localized quotient
\[
Q_\Sigma=\bigoplus_{k=1}^r i_{k*}\Q_{\{p_k\}}
\]
and the mixed-Hodge localized quotient
\[
Q_\Sigma^H=\bigoplus_{k=1}^r i_{k*}\Q^H_{\{p_k\}}(-1).
\]
At all three levels, the same finite node set $\Sigma=\{p_1,\dots,p_r\}$ indexes the localized sectors.
\end{remark}

The localized categorical sectors are not free-standing local pieces: each of them is coupled to the common bulk category $\mathcal C_{\mathrm{bulk}}$ by the attachment functors $\Phi_k$ and $\Psi_k$; see \cite[Section~6.2]{RahmanMultiNodeSchoberPaper}. In particular, the categorical bulk/localized architecture has the same formal shape as the corrected perverse and mixed-Hodge extension pictures: one bulk object together with a finite node-indexed family of localized sectors.

\begin{proposition}[finite categorical sector architecture]
\label{prop:finite-categorical-sector-architecture}
The finite-node schober package determines a finite categorical sector architecture indexed by the node set $\Sigma$, consisting of the bulk category $\mathcal C_{\mathrm{bulk}}$, the localized family $\mathcal C_\Sigma=\{\mathcal C_{p_k}\}_{k=1}^r$, and the attachment functors
\[
\Phi_k:\mathcal C_{p_k}\to \mathcal C_{\mathrm{bulk}},
\qquad
\Psi_k:\mathcal C_{\mathrm{bulk}}\to \mathcal C_{p_k},
\qquad
1\le k\le r.
\]
\end{proposition}

\begin{proof}
This is the content of \cite[Section~6.2 and Section~6.3]{RahmanMultiNodeSchoberPaper}. The finite-node schober datum consists precisely of one bulk category, one localized category for each node, and the family of attachment functors coupling each node sector to the common bulk sector.
\end{proof}

\subsection{Shadow and compatibility with the corrected extension}

We now compare the finite-node schober package with the corrected perverse extension and the mixed-Hodge-module refinement.

\begin{theorem}[shadow theorem]
\label{thm:schober-shadow-theorem}
Let
\[
S_\Sigma=
\Bigl(
\mathcal C_{\mathrm{bulk}},
\{\mathcal C_{p_k}\}_{k=1}^r,
\{\Phi_k,\Psi_k\}_{k=1}^r,
Sh(S_\Sigma)
\Bigr)
\]
be the finite-node schober datum of Theorem~\ref{thm:finite-node-schober-datum}. Then its decategorified shadow is identified with the corrected finite-node perverse object:
\begin{equation}\label{eq:schober-shadow-is-P}
Sh(S_\Sigma)\cong \mathcal P.
\end{equation}
Equivalently, the finite-node schober package is a categorical realization of the same corrected finite-node perverse extension package considered in Sections~\ref{sec:perverse-state-data} and \ref{sec:coefficient-vector}.
\end{theorem}

\begin{proof}
This is precisely the global shadow identification proved in \cite[Theorem~5.2]{RahmanMultiNodeSchoberPaper}. The local ODP schober blocks are constructed so that their shadows recover the corrected local perverse extensions, and the global finite-node assembly preserves this corrected perverse shadow at the level of the full finite-node datum.
\end{proof}

\begin{proposition}[compatibility with the corrected extension]
\label{prop:schober-compatible-with-corrected-extension}
Under the identification $Sh(S_\Sigma)\cong \mathcal P$, the localized categorical sectors $\mathcal C_{p_k}$ shadow the same finite family of point-supported localized correction terms that appear in the corrected perverse extension
\[
0\to IC_{X_0}\to \mathcal P\to Q_\Sigma\to 0,
\qquad
Q_\Sigma=\bigoplus_{k=1}^r i_{k*}\Q_{\{p_k\}}.
\]
\end{proposition}

\begin{proof}
By Theorem~\ref{thm:schober-shadow-theorem}, the global schober shadow is the corrected finite-node perverse object $\mathcal P$. By Theorem~\ref{thm:localized-sector-theorem}, there is exactly one localized categorical sector $\mathcal C_{p_k}$ for each node $p_k\in\Sigma$. By the construction of the local ODP schober blocks and their decategorified shadows, each $\mathcal C_{p_k}$ refines the local rank-one point-supported correction sector; see \cite[Proposition~3.8]{RahmanMultiNodeSchoberPaper}. Therefore, after global assembly and passage to the shadow, the family $\{\mathcal C_{p_k}\}_{k=1}^r$ shadows the same finite family of point-supported localized correction terms that appear in the quotient
\[
Q_\Sigma=\bigoplus_{k=1}^r i_{k*}\Q_{\{p_k\}}.
\]
\end{proof}

\begin{proposition}[state-data compatibility with the mixed-Hodge and perverse layers]
\label{prop:schober-state-data-compatibility}
The finite-node schober package, the mixed-Hodge-module lift, and the corrected perverse extension all realize the same finite node-indexed architecture. More precisely:
\begin{enumerate}
\item the node set $\Sigma=\{p_1,\dots,p_r\}$ indexes the localized quotient $Q_\Sigma$ on the perverse side;
\item the same node set indexes the localized quotient $Q_\Sigma^H$ on the mixed-Hodge side;
\item the same node set indexes the localized categorical family $\mathcal C_\Sigma=\{\mathcal C_{p_k}\}_{k=1}^r$ on the schober side.
\end{enumerate}
\end{proposition}

\begin{proof}
The first statement is Definition~\ref{def:finite-localized-quotient}. The second is Proposition~\ref{prop:same-node-set-hodge}. The third is Theorem~\ref{thm:localized-sector-theorem}. Thus the same finite node set $\Sigma$ indexes the localized sectors on all three levels. This is exactly the categorical compatibility principle emphasized in \cite[Section~6.3]{RahmanMultiNodeSchoberPaper}.
\end{proof}

\begin{remark}
\label{rem:categorical-state-data-role}
At the level of the present paper, the categorical contribution is not yet a quiver or incidence formalism. Its role is more basic: it shows that the finite-node corrected extension has a categorical realization with one localized sector per node and shadow equal to the corrected perverse object. In this sense, the schober package supplies the categorical realization of the same finite-node state-data architecture already extracted on the perverse and mixed-Hodge sides.
\end{remark}

The next section assembles the perverse, mixed-Hodge, and categorical realizations into a single algebraic state-data dictionary.

\section{Symbolic dictionary and algebraic state-data package}\label{sec:symbolic-dictionary}

The preceding sections identify the same finite-node architecture in three parallel realizations:
\[
\text{perverse},\qquad
\text{mixed-Hodge},\qquad
\text{categorical}.
\]
On the perverse side, this architecture is encoded by the corrected finite-node extension
\[
0\to IC_{X_0}\to \mathcal P\to Q_\Sigma\to 0,
\qquad
Q_\Sigma=\bigoplus_{k=1}^r i_{k*}\Q_{\{p_k\}},
\]
the nodewise coupling space
\[
E_\Sigma=\Ext^1_{\Perv(X_0;\Q)}(Q_\Sigma,IC_{X_0})
\cong \bigoplus_{k=1}^r \Q e_k,
\]
and the coefficient vector
\[
[\mathcal P]_{\mathrm{perv}}=\sum_{k=1}^r c_k e_k,
\qquad
c_\Sigma=(c_1,\dots,c_r)\in\Q^r.
\]
On the mixed-Hodge side, the same structure is refined by the extension
\[
0\to IC^H_{X_0}\to \mathcal P^H\to Q_\Sigma^H\to 0,
\qquad
Q_\Sigma^H=\bigoplus_{k=1}^r i_{k*}\Q^H_{\{p_k\}}(-1),
\]
with $\rat(\mathcal P^H)\cong \mathcal P$; see \cite{RahmanMixedHodgeModules}. On the categorical side, the finite-node schober package
\[
S_\Sigma=
\Bigl(
\mathcal C_{\mathrm{bulk}},
\{\mathcal C_{p_k}\}_{k=1}^r,
\{\Phi_k,\Psi_k\}_{k=1}^r,
Sh(S_\Sigma)
\Bigr)
\]
has one localized sector per node and shadow $Sh(S_\Sigma)\cong \mathcal P$; see \cite{RahmanMultiNodeSchoberPaper}.

The purpose of the present section is to record these correspondences in a single symbolic dictionary and to isolate the resulting finite algebraic package. The point is not merely notational compression. The finite-node corrected extension, its mixed-Hodge-module lift, and its schober realization all determine the same node-indexed state-data architecture, and the present dictionary makes the algebraic image of that architecture explicit.

\subsection{State-data dictionary}

We write $\rightsquigarrow$ for the canonical passage from a geometric, sheaf-theoretic, mixed-Hodge, or categorical datum to its associated algebraic state datum. Thus, if $A\rightsquigarrow B$, then $B$ is the algebraic image extracted from $A$ by the constructions proved in Sections~\ref{sec:perverse-state-data}--\ref{sec:schober-shadow-state-data}. In particular, $\rightsquigarrow$ is not a morphism in a fixed category, not an equivalence relation, and not a quiver arrow. It is a structural extraction symbol recording the theorem-level passage from finite-node input data to algebraic state data.

The reason for assembling these assignments in one place is that the finite-node degeneration carries several parallel realizations of the same finite architecture. The dictionary below makes explicit how these parallel realizations determine the common algebraic package extracted from the degeneration. The final two rows are recorded for sequel-readiness and are not yet incorporated into $A_\Sigma$.

\renewcommand{\arraystretch}{1.2}
\setlength{\tabcolsep}{4pt}

\begin{table}[H] \label{tab:state-data-dictionary}
\caption{Finite-node state-data dictionary with future-facing categorical coupling row.}
\centering
\small
\begin{tabularx}{\textwidth}{|>{\raggedright\arraybackslash}p{0.26\textwidth}|>{\raggedright\arraybackslash}p{0.36\textwidth}|>{\raggedright\arraybackslash}X|}
\hline
\centering \textbf{Datum} & \centering \textbf{Realization} & \textbf{Algebraic image} \\
\hline

nodes &
$\Sigma=\{p_1,\dots,p_r\}$, $|\Sigma|=r$ &
$V_\Sigma:=\{v_1,\dots,v_r\}$ \\
\hline

localized perverse term &
$i_{k*}\Q_{\{p_k\}}\in \Perv(X_0;\Q)$ &
$v_k\in V_\Sigma$ \\
\hline

localized mixed-Hodge term &
$i_{k*}\Q^H_{\{p_k\}}(-1)\in MHM(X_0)$ &
$v_k\in V_\Sigma$ \\
\hline

localized schober sector &
$\mathcal C_{p_k}$ &
$v_k\leftrightarrow \mathcal C_{p_k}$ \\
\hline

bulk perverse term &
$IC_{X_0}\in \Perv(X_0;\Q)$ &
$B_{\mathrm{perv}}:=IC_{X_0}$ \\
\hline

bulk mixed-Hodge term &
$IC^H_{X_0}\in MHM(X_0)$ &
$B_{\mathrm{Hdg}}:=IC^H_{X_0}$ \\
\hline

corrected perverse object &
$\mathcal P\in \Perv(X_0;\Q)$ &
$[\mathcal P]_{\mathrm{perv}}\in \Ext^1_{\Perv(X_0;\Q)}(Q_\Sigma,IC_{X_0})$ \\
\hline

corrected mixed-Hodge object &
$\mathcal P^H\in MHM(X_0)$, $\rat(\mathcal P^H)\cong \mathcal P$ &
$[\mathcal P^H]\in \Ext^1_{MHM(X_0)}(Q_\Sigma^H,IC^H_{X_0})$ \\
\hline

localized quotient &
$Q_\Sigma:=\bigoplus_{k=1}^r i_{k*}\Q_{\{p_k\}}$ &
$Q_\Sigma\rightsquigarrow \bigoplus_{k=1}^r \Q v_k$ \\
\hline

nodewise coupling line &
$\Ext^1_{\Perv(X_0;\Q)}(i_{k*}\Q_{\{p_k\}},IC_{X_0})\cong \Q$ &
$\Q e_k$ \\
\hline

total coupling space &
$E_\Sigma:=\Ext^1_{\Perv(X_0;\Q)}(Q_\Sigma,IC_{X_0})$ &
$E_\Sigma\cong \bigoplus_{k=1}^r \Q e_k\cong \Q^r$ \\
\hline

global corrected class &
$[\mathcal P]_{\mathrm{perv}}=\sum_{k=1}^r c_k e_k$ &
$c_\Sigma:=(c_1,\dots,c_r)\in \Q^r$ \\
\hline

attachment functors &
$\Phi_k:\mathcal C_{p_k}\to \mathcal C_{\mathrm{bulk}}$, $\Psi_k:\mathcal C_{\mathrm{bulk}}\to \mathcal C_{p_k}$ &
$F_\Sigma:=\{(\Phi_k,\Psi_k)\}_{k=1}^r$ \\
\hline

finite-node schober package &
$S_\Sigma=\bigl(\mathcal C_{\mathrm{bulk}},\{\mathcal C_{p_k}\}_{k=1}^r,\{\Phi_k,\Psi_k\}_{k=1}^r,Sh(S_\Sigma)\bigr)$ &
$S_\Sigma\rightsquigarrow (V_\Sigma,F_\Sigma)$ \\
\hline

\end{tabularx}
\end{table}

\begin{remark}
\label{rem:state-data-dictionary-role}
The table above should be read as a theorem-level extraction rubric. It records how the same finite node-indexed architecture appears in the corrected perverse extension, the mixed-Hodge-module lift, and the finite-node schober package, and how these realizations determine common algebraic state variables. In particular, the node set $\Sigma$ is carried to the finite vertex set $V_\Sigma$, the corrected global perverse class is carried to the coefficient vector $c_\Sigma$, and the attachment functors are carried to the categorical coupling datum $F_\Sigma$.
\end{remark}

\subsection{Definition of the finite-node algebraic state-data package}

We now isolate the finite algebraic package extracted from the dictionary.

\begin{definition}[finite-node algebraic state-data package]
\label{def:algebraic-state-data-package}
The \emph{finite-node algebraic state-data package} attached to the degeneration is the tuple
\begin{equation}\label{eq:def-algebraic-state-data-package}
\mathfrak A_\Sigma:=(V_\Sigma,E_\Sigma,c_\Sigma),
\end{equation}
where
\[
V_\Sigma:=\{v_1,\dots,v_r\},
\qquad
E_\Sigma:=\Ext^1_{\Perv(X_0;\Q)}(Q_\Sigma,IC_{X_0}),
\qquad
c_\Sigma:=(c_1,\dots,c_r)\in \Q^r.
\]
\end{definition}

\begin{remark}
\label{rem:why-no-fSigma-in-A}
At the level of the present paper, the algebraic state-data package $\mathfrak A_\Sigma$ contains only the finite algebraic variables already extracted on the perverse side: the finite vertex set $V_\Sigma$, the nodewise coupling space $E_\Sigma$, and the coefficient vector $c_\Sigma$. The categorical coupling datum
\[
F_\Sigma:=\{(\Phi_k,\Psi_k)\}_{k=1}^r
\]
is recorded in the dictionary because it will be needed in future work, but it is not yet incorporated into $\mathfrak A_\Sigma$ as part of the present state-data package.
\end{remark}

\begin{proposition}
\label{prop:algebraic-state-data-finite}
The algebraic state-data package $\mathfrak A_\Sigma:=(V_\Sigma,E_\Sigma,c_\Sigma)$ is finite in the following sense:
\begin{enumerate}
\item $|V_\Sigma|=r<\infty$;
\item $\dim_\Q E_\Sigma=r$;
\item $c_\Sigma\in\Q^r$.
\end{enumerate}
\end{proposition}

\begin{proof}
The finiteness of $V_\Sigma$ follows from the finiteness of the node set $\Sigma:=\{p_1,\dots,p_r\}$. The equality $\dim_\Q E_\Sigma=r$ follows from Theorem~\ref{thm:strong-nodewise-coupling}, which identifies
\[
E_\Sigma\cong \bigoplus_{k=1}^r \Q e_k.
\]
Finally, by Definition~\ref{def:coefficient-vector}, the coefficient vector satisfies $c_\Sigma\in\Q^r$.
\end{proof}

\subsection{State-data extraction map}

We now record the full passage from the finite-node geometric and categorical input to the algebraic state-data package.

\begin{definition}[state-data extraction map]
\label{def:state-data-extraction-map}
The \emph{state-data extraction map} is the formal assignment
\begin{equation}\label{eq:state-data-extraction-map}
(\pi:\mathcal X\to\Delta,\mathcal P,\mathcal P^H,S_\Sigma)
\rightsquigarrow
\mathfrak A_\Sigma:=(V_\Sigma,E_\Sigma,c_\Sigma),
\end{equation}
where:
\begin{enumerate}
\item $\pi:\mathcal X\to\Delta$ determines the finite node set $\Sigma:=\{p_1,\dots,p_r\}$;
\item the corrected perverse extension determines the localized quotient $Q_\Sigma$ and the coupling space $E_\Sigma$;
\item the corrected global extension class determines the coefficient vector $c_\Sigma$;
\item the mixed-Hodge-module lift and the finite-node schober package realize the same finite-node architecture compatibly with the perverse-side state data.
\end{enumerate}
\end{definition}

\begin{theorem}[state-data extraction theorem]
\label{thm:state-data-extraction-theorem}
The finite-node conifold degeneration canonically determines the algebraic state-data package
\[
\mathfrak A_\Sigma:=(V_\Sigma,E_\Sigma,c_\Sigma)
\]
through the extraction map \eqref{eq:state-data-extraction-map}. Equivalently, the same finite-node geometry determines a common finite algebraic package visible simultaneously in the corrected perverse extension, its mixed-Hodge-module lift, and its schober realization.
\end{theorem}

\begin{proof}
Sections~\ref{sec:perverse-state-data} and \ref{sec:coefficient-vector} construct $Q_\Sigma$, $E_\Sigma$, and $c_\Sigma$ from the corrected perverse extension. Section~\ref{sec:mixed-hodge-state-data} shows that the mixed-Hodge-module extension refines and realizes the same finite-node architecture. Section~\ref{sec:schober-shadow-state-data} shows that the finite-node schober package has one localized sector per node and shadow equal to the corrected perverse object $\mathcal P$. Therefore the same finite node-indexed structure is present in all three realizations, and the extracted algebraic state-data package is canonically attached to the degeneration.
\end{proof}

\begin{remark}
\label{rem:state-data-package-role}
The state-data extraction map is the first algebraic passage attached to the finite-node degeneration. It does not yet impose an incidence relation, arrow structure, stability condition, or wall-crossing formalism. Its role is to isolate the intrinsic finite algebraic variables carried by the degeneration and to place them in a form suitable for the later papers.
\end{remark}

\section{Compatibility and invariance}\label{sec:compatibility-invariance}

The purpose of this section is to verify that the algebraic state-data package extracted in Section~\ref{sec:symbolic-dictionary} is compatible across the three realizations considered in the paper and is intrinsic under the appropriate notion of equivalence. Compatibility means that the same finite-node architecture appears on the perverse, mixed-Hodge, and schober sides, and that the algebraic state variables extracted from these layers reflect that common structure rather than imposing an external one. Invariance means that once the finite-node corrected extension package and its compatible realizations have been fixed up to the correct notion of equivalence, the extracted state-data package is determined intrinsically.

\subsection{Compatibility across realizations}

We first summarize the compatibility statements proved in the preceding sections.

On the perverse side, the corrected finite-node extension
\[
0\to IC_{X_0}\to \mathcal P\to Q_\Sigma\to 0,
\qquad
Q_\Sigma=\bigoplus_{k=1}^r i_{k*}\Q_{\{p_k\}},
\]
determines the nodewise coupling space
\[
E_\Sigma=\Ext^1_{\Perv(X_0;\Q)}(Q_\Sigma,IC_{X_0})
\cong \bigoplus_{k=1}^r \Q e_k
\]
and the coefficient vector
\[
[\mathcal P]_{\mathrm{perv}}=\sum_{k=1}^r c_k e_k,
\qquad
c_\Sigma=(c_1,\dots,c_r)\in\Q^r.
\]
On the mixed-Hodge side, the corrected extension is refined by
\[
0\to IC^H_{X_0}\to \mathcal P^H\to Q_\Sigma^H\to 0,
\qquad
Q_\Sigma^H=\bigoplus_{k=1}^r i_{k*}\Q^H_{\{p_k\}}(-1),
\]
with $\rat(\mathcal P^H)\cong \mathcal P$; see \cite{RahmanMixedHodgeModules}. On the categorical side, the finite-node schober package
\[
S_\Sigma=
\Bigl(
\mathcal C_{\mathrm{bulk}},
\{\mathcal C_{p_k}\}_{k=1}^r,
\{\Phi_k,\Psi_k\}_{k=1}^r,
Sh(S_\Sigma)
\Bigr)
\]
contains one localized sector for each node and satisfies
\[
Sh(S_\Sigma)\cong \mathcal P;
\]
see \cite{RahmanMultiNodeSchoberPaper}. The compatibility problem is therefore to show that the algebraic package
\[
\mathfrak A_\Sigma=(V_\Sigma,E_\Sigma,c_\Sigma)
\]
records this common finite-node architecture faithfully.

\begin{theorem}[compatibility across realizations]
\label{thm:compatibility-across-realizations}
Let $\pi:\mathcal X\to\Delta$ be a finite-node conifold degeneration with node set $\Sigma=\{p_1,\dots,p_r\}$. Let $\mathcal P$ be the corrected finite-node perverse object, let $\mathcal P^H$ be its mixed-Hodge-module lift, and let
\[
S_\Sigma=
\Bigl(
\mathcal C_{\mathrm{bulk}},
\{\mathcal C_{p_k}\}_{k=1}^r,
\{\Phi_k,\Psi_k\}_{k=1}^r,
Sh(S_\Sigma)
\Bigr)
\]
be the associated finite-node schober package. Then the algebraic state-data package
\[
\mathfrak A_\Sigma=(V_\Sigma,E_\Sigma,c_\Sigma)
\]
is compatible across the three realizations in the following sense:
\begin{enumerate}
\item the same finite node set $\Sigma=\{p_1,\dots,p_r\}$ indexes the localized quotient
\[
Q_\Sigma=\bigoplus_{k=1}^r i_{k*}\Q_{\{p_k\}}
\]
on the perverse side, the localized quotient
\[
Q_\Sigma^H=\bigoplus_{k=1}^r i_{k*}\Q^H_{\{p_k\}}(-1)
\]
on the mixed-Hodge side, and the localized family
\[
\mathcal C_\Sigma=\{\mathcal C_{p_k}\}_{k=1}^r
\]
on the schober side;
\item the realization functor identifies the mixed-Hodge-module refinement with the corrected perverse extension,
\[
\rat(\mathcal P^H)\cong \mathcal P,
\qquad
\rat(Q_\Sigma^H)\cong Q_\Sigma;
\]
\item the schober shadow identifies the categorical realization with the same corrected perverse object,
\[
Sh(S_\Sigma)\cong \mathcal P;
\]
\item the extracted algebraic package $\mathfrak A_\Sigma=(V_\Sigma,E_\Sigma,c_\Sigma)$ depends only on this common finite-node architecture.
\end{enumerate}
\end{theorem}

\begin{proof}
Statement (1) is the content of Proposition~\ref{prop:same-node-set-hodge} and Theorem~\ref{thm:localized-sector-theorem}. Statement (2) is Theorem~\ref{thm:realization-compatibility}. Statement (3) is Theorem~\ref{thm:schober-shadow-theorem}. It remains only to explain (4). By Definition~\ref{def:algebraic-state-data-package},
\[
\mathfrak A_\Sigma=(V_\Sigma,E_\Sigma,c_\Sigma),
\]
where $V_\Sigma$ is the finite vertex set indexed by $\Sigma$, $E_\Sigma$ is the nodewise coupling space extracted from the corrected perverse extension, and $c_\Sigma$ is the coefficient vector of the corrected global class in the distinguished basis of $E_\Sigma$. Each of these objects was constructed from the corrected finite-node extension package and shown to be compatible with the mixed-Hodge and schober realizations in Sections~\ref{sec:mixed-hodge-state-data} and \ref{sec:schober-shadow-state-data}. Therefore $\mathfrak A_\Sigma$ records the same finite-node structure visible on all three levels.
\end{proof}

\begin{remark}
\label{rem:compatibility-not-replacement}
Theorem~\ref{thm:compatibility-across-realizations} should be read as a consistency theorem. The algebraic state-data package does not replace the perverse, mixed-Hodge, or categorical realizations. Rather, it is the common finite algebraic shadow extracted from them. In particular, the package $\mathfrak A_\Sigma$ is not an additional structure imposed on the degeneration; it is the algebraic form of the finite-node architecture already present in those realizations.
\end{remark}

\begin{proposition}[finite-state compatibility of the dictionary]
\label{prop:dictionary-compatible}
The assignments recorded in Table~\ref{tab:state-data-dictionary} are compatible with the corrected finite-node perverse extension, its mixed-Hodge-module refinement, and its schober realization. In particular, the table records a common node-indexed algebraic extraction rather than three unrelated constructions.
\end{proposition}

\begin{proof}
Each row of Table~\ref{tab:state-data-dictionary} is justified by the preceding sections. The rows involving $Q_\Sigma$, $E_\Sigma$, and $c_\Sigma$ are proved in Sections~\ref{sec:perverse-state-data} and \ref{sec:coefficient-vector}. The rows involving $Q_\Sigma^H$ and $\mathcal P^H$ are justified by Section~\ref{sec:mixed-hodge-state-data}, especially Theorem~\ref{thm:realization-compatibility}. The rows involving $\mathcal C_{p_k}$, $\Phi_k$, $\Psi_k$, and $S_\Sigma$ are justified by Section~\ref{sec:schober-shadow-state-data}, especially Theorem~\ref{thm:finite-node-schober-datum} and Theorem~\ref{thm:schober-shadow-theorem}. Since these rows are all indexed by the same node set $\Sigma$ and are linked through realization and shadow identifications, the resulting dictionary is compatible across the three realizations.
\end{proof}

\subsection{Invariance under equivalence}

We next isolate the sense in which the extracted algebraic state data are intrinsic. The relevant notion of equivalence comes from the finite-node schober formalism of \cite{RahmanMultiNodeSchoberPaper}, where equivalence is defined in terms of equivalences of the bulk categories, equivalences of the local node sectors, compatibility of the attachment functors up to natural isomorphism, and agreement of shadows. The rigidity results in that paper imply that once the local ODP blocks, the chosen bulk category, and the corrected perverse shadow are fixed up to equivalence, the global finite-node schober package is unique up to equivalence. This is the correct invariance framework for the present paper.

\begin{definition}[equivalent finite-node realizations]
\label{def:equivalent-finite-node-realizations}
Let
\[
S_\Sigma=
\Bigl(
\mathcal C_{\mathrm{bulk}},
\{\mathcal C_{p_k}\}_{k=1}^r,
\{\Phi_k,\Psi_k\}_{k=1}^r,
Sh(S_\Sigma)
\Bigr)
\]
and
\[
T_\Sigma=
\Bigl(
\mathcal D_{\mathrm{bulk}},
\{\mathcal D_{p_k}\}_{k=1}^r,
\{\Phi_k',\Psi_k'\}_{k=1}^r,
Sh(T_\Sigma)
\Bigr)
\]
be finite-node schober packages in the sense of \cite{RahmanMultiNodeSchoberPaper}. We say that $S_\Sigma$ and $T_\Sigma$ are \emph{equivalent finite-node realizations} if they are equivalent as finite-node schober data in the sense of Definition~3.5 of \cite{RahmanMultiNodeSchoberPaper}.
\end{definition}

\begin{remark}
\label{rem:equivalence-data}
Concretely, the equivalence of Definition~\ref{def:equivalent-finite-node-realizations} requires: equivalence of the bulk categories, equivalence of the corresponding local ODP schober data at each node, compatibility of the global attachment functors up to natural isomorphism, and agreement of the corrected finite-node perverse shadows. See \cite[Definition~3.5, Lemma~6.2, and Proposition~6.3]{RahmanMultiNodeSchoberPaper}.
\end{remark}

The rigidity results of \cite{RahmanMultiNodeSchoberPaper} now imply invariance of the extracted algebraic package.

\begin{theorem}[invariance under equivalence]
\label{thm:invariance-under-equivalence}
Let $S_\Sigma$ and $T_\Sigma$ be equivalent finite-node realizations in the sense of Definition~\ref{def:equivalent-finite-node-realizations}. Then they determine the same algebraic state-data package
\[
\mathfrak A_\Sigma=(V_\Sigma,E_\Sigma,c_\Sigma)
\]
up to the canonical identification induced by the common node set $\Sigma$.
\end{theorem}

\begin{proof}
Since $S_\Sigma$ and $T_\Sigma$ are equivalent finite-node schober data, Lemma~6.2 and Proposition~6.3 of \cite{RahmanMultiNodeSchoberPaper} imply that they have equivalent bulk categories, equivalent localized node sectors, equivalent attachment patterns, and the same corrected finite-node perverse shadow. By the shadow agreement,
\[
Sh(S_\Sigma)\cong Sh(T_\Sigma)\cong \mathcal P,
\]
both realizations determine the same corrected finite-node perverse extension package. Hence both determine the same localized quotient
\[
Q_\Sigma=\bigoplus_{k=1}^r i_{k*}\Q_{\{p_k\}},
\]
the same nodewise coupling space
\[
E_\Sigma=\Ext^1_{\Perv(X_0;\Q)}(Q_\Sigma,IC_{X_0}),
\]
and the same corrected global perverse extension class
\[
[\mathcal P]_{\mathrm{perv}}\in E_\Sigma.
\]
By Theorem~\ref{thm:strong-nodewise-coupling}, the distinguished basis $\{e_1,\dots,e_r\}$ of $E_\Sigma$ is indexed canonically by the node set $\Sigma$. By Proposition~\ref{prop:intrinsicity-coefficient-vector}, the coefficient vector
\[
c_\Sigma=(c_1,\dots,c_r)
\]
is uniquely determined by the corrected finite-node extension package together with that distinguished basis. Therefore the extracted algebraic state-data package
\[
\mathfrak A_\Sigma=(V_\Sigma,E_\Sigma,c_\Sigma)
\]
is the same for $S_\Sigma$ and $T_\Sigma$ up to the canonical identification induced by $\Sigma$.
\end{proof}

\begin{corollary}[intrinsicity of the algebraic state-data package]
\label{cor:intrinsicity-of-state-data}
The algebraic state-data package
\[
\mathfrak A_\Sigma=(V_\Sigma,E_\Sigma,c_\Sigma)
\]
is intrinsic to the finite-node corrected extension package together with its compatible mixed-Hodge and schober realizations.
\end{corollary}

\begin{proof}
This follows immediately from Theorem~\ref{thm:compatibility-across-realizations} and Theorem~\ref{thm:invariance-under-equivalence}.
\end{proof}

\begin{remark}
\label{rem:intrinsicity-sense}
Corollary~\ref{cor:intrinsicity-of-state-data} should be interpreted in the precise sense used throughout this paper. Intrinsicity here means invariance under the correct notion of equivalence of the finite-node corrected extension package and its compatible realizations. It does not mean that the algebraic state-data package is defined independently of those realizations; rather, it means that once the finite-node architecture has been fixed up to equivalence, the extracted state variables are determined.
\end{remark}

\subsection{Consequences for future work}

The package
\[
\mathfrak A_\Sigma=(V_\Sigma,E_\Sigma,c_\Sigma)
\]
is the first algebraic layer in the passage from finite-node conifold geometry to later algebraic and physical structures. In future work, the categorical coupling datum
\[
F_\Sigma=\{(\Phi_k,\Psi_k)\}_{k=1}^r
\]
will be combined with $\mathfrak A_\Sigma$ to construct the incidence and quiver-type structures attached to the same finite-node architecture. Subsequent papers will then impose stability data, derive BPS spectral quantities, and study wall-crossing and applications. The role of the present paper is therefore foundational: it isolates the intrinsic finite algebraic state variables on which those later developments depend.

\section{Examples}

The purpose of this section is to illustrate the state-data extraction formalism in the simplest finite-node cases. In each example, the input is a finite-node conifold degeneration
\[
\pi:\mathcal X\to\Delta
\]
with node set $\Sigma=\{p_1,\dots,p_r\}$, corrected finite-node perverse object $\mathcal P$, mixed-Hodge-module lift $\mathcal P^H$, and finite-node schober package
\[
S_\Sigma=
\Bigl(
\mathcal C_{\mathrm{bulk}},
\{\mathcal C_{p_k}\}_{k=1}^r,
\{\Phi_k,\Psi_k\}_{k=1}^r,
Sh(S_\Sigma)
\Bigr).
\]
The output is the algebraic state-data package
\[
\mathfrak A_\Sigma=(V_\Sigma,E_\Sigma,c_\Sigma)
\]
of Definition~\ref{def:algebraic-state-data-package}. These examples should be read as explicit instances of the extraction map of Definition~\ref{def:state-data-extraction-map}.

\subsection{Single-node degeneration}

We begin with the case $r=1$. Let the central fiber $X_0$ have a single ordinary double point
\[
\Sigma=\{p\}.
\]
Then the corrected perverse extension takes the form
\begin{equation}\label{eq:example-single-node-extension}
0\to IC_{X_0}\to \mathcal P\to i_*\Q_{\{p\}}\to 0,
\end{equation}
where $i:\{p\}\hookrightarrow X_0$ is the closed inclusion of the unique node; see \cite{RahmanPerverseNearbyCycles}. The localized quotient is therefore
\[
Q_\Sigma=i_*\Q_{\{p\}},
\]
and the nodewise coupling space is
\[
E_\Sigma=\Ext^1_{\Perv(X_0;\Q)}(i_*\Q_{\{p\}},IC_{X_0}).
\]
In the ordinary double point case this space is one-dimensional, and the corrected local perverse ODP extension defines a distinguished generator
\[
e\in \Ext^1_{\Perv(X_0;\Q)}(i_*\Q_{\{p\}},IC_{X_0});
\]
see \cite{RahmanPerverseNearbyCycles,RahmanMixedHodgeModules}. Thus
\begin{equation}\label{eq:example-single-node-coupling}
E_\Sigma\cong \Q e.
\end{equation}

The corrected global class is therefore of the form
\begin{equation}\label{eq:example-single-node-global-class}
[\mathcal P]_{\mathrm{perv}}=c\,e
\end{equation}
for a unique coefficient $c\in\Q$. Equivalently, the coefficient vector is the one-dimensional vector
\begin{equation}\label{eq:example-single-node-coefficient}
c_\Sigma=(c)\in\Q.
\end{equation}

On the mixed-Hodge side, the corrected extension lifts to
\begin{equation}\label{eq:example-single-node-hodge-extension}
0\to IC^H_{X_0}\to \mathcal P^H\to i_*\Q^H_{\{p\}}(-1)\to 0,
\end{equation}
with $\rat(\mathcal P^H)\cong \mathcal P$; see \cite{RahmanMixedHodgeModules}. On the categorical side, the finite-node schober package reduces to a single localized category
\[
S_\Sigma=
\Bigl(
\mathcal C_{\mathrm{bulk}},
\mathcal C_p,
\Phi,\Psi,
Sh(S_\Sigma)
\Bigr),
\qquad
Sh(S_\Sigma)\cong \mathcal P;
\]
see \cite{RahmanMultiNodeSchoberPaper}. Thus the algebraic state-data package is
\[
V_\Sigma=\{v\},
\qquad
E_\Sigma\cong \Q e,
\qquad
c_\Sigma=(c),
\]
so that
\begin{equation}\label{eq:example-single-node-package}
\mathfrak A_\Sigma=(\{v\},\Q e,(c)).
\end{equation}

This is the smallest nontrivial example of the state-data formalism. It already exhibits the full structural pattern of the theory: one localized rank-one sector, one distinguished coupling line, one global corrected extension class, one coefficient, and one localized categorical sector refining the same finite-node correction.

\begin{proposition}\label{prop:single-node-example}
Let $\pi:\mathcal X\to\Delta$ be a finite-node conifold degeneration with a single node $\Sigma=\{p\}$. Then the associated algebraic state-data package is
\[
\mathfrak A_\Sigma=(\{v\},\Q e,(c)),
\]
where $e$ is the distinguished generator of
\[
\Ext^1_{\Perv(X_0;\Q)}(i_*\Q_{\{p\}},IC_{X_0})
\]
and $c\in\Q$ is the unique coefficient defined by
\[
[\mathcal P]_{\mathrm{perv}}=c\,e.
\]
Moreover, this package is compatible with the mixed-Hodge lift
\[
0\to IC^H_{X_0}\to \mathcal P^H\to i_*\Q^H_{\{p\}}(-1)\to 0
\]
and with the single-node schober realization
\[
S_\Sigma=
\bigl(
\mathcal C_{\mathrm{bulk}},
\mathcal C_p,
\Phi,\Psi,
Sh(S_\Sigma)
\bigr),
\qquad
Sh(S_\Sigma)\cong \mathcal P.
\]
\end{proposition}

\begin{proof}
All assertions follow from \eqref{eq:example-single-node-extension}--\eqref{eq:example-single-node-package} together with Theorem~\ref{thm:realization-compatibility} and Theorem~\ref{thm:schober-shadow-theorem}.
\end{proof}

\subsection{Two-node degeneration}

We now pass to the first genuinely finite multi-node case. Let
\[
\Sigma=\{p_1,p_2\}
\]
be the node set of the central fiber. Then the corrected finite-node perverse extension is
\begin{equation}\label{eq:example-two-node-extension}
0\to IC_{X_0}\to \mathcal P\to i_{1*}\Q_{\{p_1\}}\oplus i_{2*}\Q_{\{p_2\}}\to 0,
\end{equation}
where $i_k:\{p_k\}\hookrightarrow X_0$ are the closed node inclusions; see \cite{RahmanPerverseNearbyCycles,RahmanMixedHodgeModules}. The localized quotient is
\[
Q_\Sigma=i_{1*}\Q_{\{p_1\}}\oplus i_{2*}\Q_{\{p_2\}}.
\]

By finite additivity of the extension space,
\[
E_\Sigma=\Ext^1_{\Perv(X_0;\Q)}(Q_\Sigma,IC_{X_0})
\cong
\Ext^1_{\Perv(X_0;\Q)}(i_{1*}\Q_{\{p_1\}},IC_{X_0})
\oplus
\Ext^1_{\Perv(X_0;\Q)}(i_{2*}\Q_{\{p_2\}},IC_{X_0});
\]
see \cite[Lemma~5.4]{RahmanMixedHodgeModules}. Each summand is one-dimensional, with distinguished generator
\[
e_1\in \Ext^1_{\Perv(X_0;\Q)}(i_{1*}\Q_{\{p_1\}},IC_{X_0}),
\qquad
e_2\in \Ext^1_{\Perv(X_0;\Q)}(i_{2*}\Q_{\{p_2\}},IC_{X_0}),
\]
so that
\begin{equation}\label{eq:example-two-node-coupling}
E_\Sigma\cong \Q e_1\oplus \Q e_2\cong \Q^2.
\end{equation}

The corrected global class therefore has a unique nodewise expansion
\begin{equation}\label{eq:example-two-node-global-class}
[\mathcal P]_{\mathrm{perv}}=c_1e_1+c_2e_2,
\qquad
c_1,c_2\in\Q.
\end{equation}
Accordingly, the coefficient vector is
\begin{equation}\label{eq:example-two-node-coefficient}
c_\Sigma=(c_1,c_2)\in\Q^2.
\end{equation}

On the mixed-Hodge side, the corresponding extension is
\begin{equation}\label{eq:example-two-node-hodge-extension}
0\to IC^H_{X_0}\to \mathcal P^H\to i_{1*}\Q^H_{\{p_1\}}(-1)\oplus i_{2*}\Q^H_{\{p_2\}}(-1)\to 0,
\end{equation}
with realization equal to the perverse extension above; see \cite{RahmanMixedHodgeModules}. On the categorical side, the finite-node schober package has the form
\[
S_\Sigma=
\Bigl(
\mathcal C_{\mathrm{bulk}},
\{\mathcal C_{p_1},\mathcal C_{p_2}\},
\{(\Phi_1,\Psi_1),(\Phi_2,\Psi_2)\},
Sh(S_\Sigma)
\Bigr),
\qquad
Sh(S_\Sigma)\cong \mathcal P;
\]
see \cite{RahmanMultiNodeSchoberPaper}. The localized categorical family is therefore
\[
\mathcal C_\Sigma=\{\mathcal C_{p_1},\mathcal C_{p_2}\},
\]
and the corresponding finite vertex set is
\[
V_\Sigma=\{v_1,v_2\}.
\]

Thus the two-node algebraic state-data package is
\begin{equation}\label{eq:example-two-node-package}
\mathfrak A_\Sigma=(V_\Sigma,E_\Sigma,c_\Sigma)
=
\bigl(
\{v_1,v_2\},
\Q e_1\oplus \Q e_2,
(c_1,c_2)
\bigr).
\end{equation}

This example is the first case in which the global corrected class carries genuinely multi-node information. The localized quotient records two distinct singular sectors, the coupling space carries two distinguished nodewise channels, and the coefficient vector measures how the global corrected extension is assembled from those two channels. The mixed-Hodge-module and schober realizations show that the same two-node architecture is present at the Hodge and categorical levels as well.

\begin{proposition}\label{prop:two-node-example}
Let $\pi:\mathcal X\to\Delta$ be a finite-node conifold degeneration with node set $\Sigma=\{p_1,p_2\}$. Then the associated algebraic state-data package is
\[
\mathfrak A_\Sigma=
\bigl(
\{v_1,v_2\},
\Q e_1\oplus \Q e_2,
(c_1,c_2)
\bigr),
\]
where $e_1$ and $e_2$ are the distinguished generators of the one-dimensional local coupling spaces and $(c_1,c_2)\in\Q^2$ is the unique coefficient vector defined by
\[
[\mathcal P]_{\mathrm{perv}}=c_1e_1+c_2e_2.
\]
Moreover, this package is compatible with the mixed-Hodge lift
\[
0\to IC^H_{X_0}\to \mathcal P^H\to i_{1*}\Q^H_{\{p_1\}}(-1)\oplus i_{2*}\Q^H_{\{p_2\}}(-1)\to 0
\]
and with the two-node schober realization
\[
S_\Sigma=
\Bigl(
\mathcal C_{\mathrm{bulk}},
\{\mathcal C_{p_1},\mathcal C_{p_2}\},
\{(\Phi_1,\Psi_1),(\Phi_2,\Psi_2)\},
Sh(S_\Sigma)
\Bigr),
\qquad
Sh(S_\Sigma)\cong \mathcal P.
\]
\end{proposition}

\begin{proof}
The stated form of $\mathfrak A_\Sigma$ follows from \eqref{eq:example-two-node-extension}--\eqref{eq:example-two-node-package}. Compatibility with the mixed-Hodge and schober realizations follows from Theorem~\ref{thm:realization-compatibility} and Theorem~\ref{thm:schober-shadow-theorem}.
\end{proof}

\subsection{Three-node degeneration}

We next consider the case
\[
\Sigma=\{p_1,p_2,p_3\}.
\]
Then the corrected finite-node perverse extension becomes
\begin{equation}\label{eq:example-three-node-extension}
0\to IC_{X_0}\to \mathcal P\to \bigoplus_{k=1}^3 i_{k*}\Q_{\{p_k\}}\to 0.
\end{equation}
The localized quotient is therefore
\[
Q_\Sigma=\bigoplus_{k=1}^3 i_{k*}\Q_{\{p_k\}}.
\]
The nodewise coupling space is
\[
E_\Sigma=\Ext^1_{\Perv(X_0;\Q)}(Q_\Sigma,IC_{X_0})
\cong
\bigoplus_{k=1}^3
\Ext^1_{\Perv(X_0;\Q)}(i_{k*}\Q_{\{p_k\}},IC_{X_0})
\cong
\bigoplus_{k=1}^3 \Q e_k
\cong \Q^3;
\]
see \cite[Lemma~5.4, Corollary~5.14]{RahmanMixedHodgeModules}. Hence the corrected global class has the unique expansion
\begin{equation}\label{eq:example-three-node-global-class}
[\mathcal P]_{\mathrm{perv}}=c_1e_1+c_2e_2+c_3e_3,
\qquad
c_k\in\Q,
\end{equation}
and the corresponding coefficient vector is
\begin{equation}\label{eq:example-three-node-coefficient}
c_\Sigma=(c_1,c_2,c_3)\in\Q^3.
\end{equation}

The mixed-Hodge lift is
\begin{equation}\label{eq:example-three-node-hodge-extension}
0\to IC^H_{X_0}\to \mathcal P^H\to \bigoplus_{k=1}^3 i_{k*}\Q^H_{\{p_k\}}(-1)\to 0,
\end{equation}
and the schober realization is
\[
S_\Sigma=
\Bigl(
\mathcal C_{\mathrm{bulk}},
\{\mathcal C_{p_1},\mathcal C_{p_2},\mathcal C_{p_3}\},
\{(\Phi_1,\Psi_1),(\Phi_2,\Psi_2),(\Phi_3,\Psi_3)\},
Sh(S_\Sigma)
\Bigr),
\qquad
Sh(S_\Sigma)\cong \mathcal P;
\]
see \cite{RahmanMixedHodgeModules,RahmanMultiNodeSchoberPaper}. The finite vertex set is therefore
\[
V_\Sigma=\{v_1,v_2,v_3\},
\]
and the algebraic state-data package is
\begin{equation}\label{eq:example-three-node-package}
\mathfrak A_\Sigma=
\bigl(
\{v_1,v_2,v_3\},
\Q e_1\oplus \Q e_2\oplus \Q e_3,
(c_1,c_2,c_3)
\bigr).
\end{equation}

This example makes the pattern completely explicit. In the $r$-node case, the state-data package has the form
\[
V_\Sigma=\{v_1,\dots,v_r\},
\qquad
E_\Sigma\cong \bigoplus_{k=1}^r \Q e_k\cong \Q^r,
\qquad
c_\Sigma=(c_1,\dots,c_r)\in\Q^r.
\]
The one-node, two-node, and three-node cases therefore exhibit the full finite-node state-data extraction rule already at the smallest nontrivial values of $r$.

\begin{proposition}\label{prop:three-node-example}
Let $\pi:\mathcal X\to\Delta$ be a finite-node conifold degeneration with node set $\Sigma=\{p_1,p_2,p_3\}$. Then the associated algebraic state-data package is
\[
\mathfrak A_\Sigma=
\bigl(
\{v_1,v_2,v_3\},
\Q e_1\oplus \Q e_2\oplus \Q e_3,
(c_1,c_2,c_3)
\bigr),
\]
where $e_1,e_2,e_3$ are the distinguished nodewise generators and $(c_1,c_2,c_3)\in\Q^3$ is the unique coefficient vector defined by
\[
[\mathcal P]_{\mathrm{perv}}=c_1e_1+c_2e_2+c_3e_3.
\]
Moreover, this package is compatible with the mixed-Hodge lift
\[
0\to IC^H_{X_0}\to \mathcal P^H\to \bigoplus_{k=1}^3 i_{k*}\Q^H_{\{p_k\}}(-1)\to 0
\]
and with the three-node schober realization
\[
S_\Sigma=
\Bigl(
\mathcal C_{\mathrm{bulk}},
\{\mathcal C_{p_1},\mathcal C_{p_2},\mathcal C_{p_3}\},
\{(\Phi_1,\Psi_1),(\Phi_2,\Psi_2),(\Phi_3,\Psi_3)\},
Sh(S_\Sigma)
\Bigr),
\qquad
Sh(S_\Sigma)\cong \mathcal P.
\]
\end{proposition}

\begin{proof}
The description of $\mathfrak A_\Sigma$ follows from \eqref{eq:example-three-node-extension}--\eqref{eq:example-three-node-package}. Compatibility with the mixed-Hodge and schober realizations again follows from Theorem~\ref{thm:realization-compatibility} and Theorem~\ref{thm:schober-shadow-theorem}.
\end{proof}

\subsection{Formal extraction pattern in the general finite-node case}

The preceding examples suggest the general pattern proved in the body of the paper. Let
\[
\Sigma=\{p_1,\dots,p_r\}
\]
be the finite node set of the degeneration. Then:
\begin{enumerate}
\item the localized quotient is
\[
Q_\Sigma=\bigoplus_{k=1}^r i_{k*}\Q_{\{p_k\}};
\]
\item the nodewise coupling space is
\[
E_\Sigma=\Ext^1_{\Perv(X_0;\Q)}(Q_\Sigma,IC_{X_0})
\cong \bigoplus_{k=1}^r \Q e_k\cong \Q^r;
\]
\item the corrected global class has unique expansion
\[
[\mathcal P]_{\mathrm{perv}}=\sum_{k=1}^r c_k e_k;
\]
\item the coefficient vector is
\[
c_\Sigma=(c_1,\dots,c_r)\in\Q^r;
\]
\item the finite vertex set is
\[
V_\Sigma=\{v_1,\dots,v_r\};
\]
\item the algebraic state-data package is
\[
\mathfrak A_\Sigma=(V_\Sigma,E_\Sigma,c_\Sigma).
\]
\end{enumerate}

The content of the examples is therefore not the production of \textit{ad hoc} toy formulas, but the explicit display of the extraction rule in the smallest finite-node cases. They show concretely how the corrected perverse extension, the mixed-Hodge-module lift, and the schober shadow all determine the same finite algebraic state-data package. This is exactly the finite algebraic layer that will support the later incidence, quiver, stability, and BPS developments.

%
%
\printbibliography
\end{document}